\documentclass[a4paper,10pt,leqno]{amsart}
\usepackage[english]{babel}
\usepackage{verbatim}
\usepackage[leqno]{amsmath}
\usepackage{amssymb,pb-diagram,pb-xy}
\usepackage[arrow,curve,matrix,tips,2cell]{xy}
\usepackage{amscd}
\usepackage[mathscr]{euscript}
\usepackage{hyperref}
\usepackage{pb-diagram}
\usepackage{mathtools}
\theoremstyle{definition}
\numberwithin{equation}{section}
\newtheorem{teor}{Theorem}[section]
\newtheorem{defin}[teor]{Definition}
\newtheorem{ejem}[teor]{Example}
\newtheorem{lema}[teor]{Lemma}

\newtheorem{prop}[teor]{Proposition}

\newtheorem{remark}[teor]{Remark}
\newtheorem{condition}[teor]{Condition}
\def\spinc{\text{Spin}^c}
\def\C{\mathbb{C}}

\def\Z{\mathbb{Z}}
\def\N{\mathbb{N}}
\def\H{\mathscr{H}}

\def\F{\mathfrak{F}}
\def\c{\mathfrak{C}}
\def\ca{\mathfrak{C}(X/A,G)}

\def\can{\mathfrak{C}^{n+1}(X/A,G)-\mathfrak{C}^{n}(X/A,G)}
\def\cn{\c^n(X/A,G)}
\def\cnn{\c^{n+1}(X/A,G)}
\def\k{\underline{\widetilde{k}}}
\def\ku{\widetilde{k}}

\def\hu{\mathbb{H}}
\def\s{\mathscr{S}}
\def\hu{\mathbb{H}}
\def\S{\mathfrak{S}}
\def\E{\mathbb{E}}
\def\D{\mathbb{D}}
\newcommand{\Ker}{\mbox{ker}}
\newcommand{\Hom}{\mbox{Hom}}
\newcommand{\Mat}{\mbox{M}}
\newcommand{\End}{\mbox{End}}
\newcommand{\FR}{\mbox{FR}}
\newcommand{\rank}{\mbox{rank}}
\newcounter{commentcounter}

\begin{document}
\title{A configuration space for equivariant connective K-homology}
\author{Mario Vel\'asquez}\thanks{The author was partially supported by Faculty of Sciences of Pontificia Universidad Javeriana through the project \emph{Aplicaciones de la K-teor\'ia en teor\'ia del \'indice y las conjeturas de isomorfismo} with ID 6552.
}
\address{ Departamento de Matem\'aticas
\\Pontificia Universidad Javeriana\\Cra. 7 No. 43-82 - Edificio Carlos Ortíz 5to piso\\ Bogot\'a D.C, Colombia}
\email{mavelasquezm@gmail.com}
\urladdr{https://sites.google.com/site/mavelasquezm/}

\date{\today}


\begin{abstract}
Following ideas of Graeme Segal, we construct an equivariant configuration space that
is a model of equivariant connective K-homology spectrum for finite groups, as a consequence we obtain an induction structure for equivariant connective K-homology. We describe explicitly the  homology with complex coefficients for the fixed points  of  this configuration space as a Hopf algebra.
\end{abstract}
\maketitle
\section{Introduction}\label{section0}
The purpose of this paper is to give a configuration space description of \textit{equivariant connective K-homology} (Proposition \ref{connectiveversion}). We describe the 
homology of the fixed point space of these configuration spaces, in terms of certain Hopf algebras studied initially by Segal in 
\cite{segal1996} and ge\-ne\-ra\-li\-zed by Wang in \cite{wang2000}. We follow ideas of Graeme Segal, and most of the results and proofs 
obtained here generalize results contained in \cite{segal1977} and 
\cite{segal1996} to the equivariant context. 
Our results answer a question posed by Wang in \cite{wang2000}. Namely, let $(M,m_0)$ be a based $G$-$\spinc$-manifold. Wang asks about the possibility to relate the Hopf algebra $$\F_G^q(M)=\bigoplus_{n\geq0}q^nK_{G\sim\mathfrak{S}_n}^*(M^n)\otimes\C,$$
to the homology of some configuration space $\c(M,m_0,G)$ thus generalizing Segal's work. Finally, as a consequence, we obtain a new model for the
equivariant  connective K-theory spectrum. The results in this paper are part of the PhD thesis of the author \cite{ve2012}.

The first appearance of configuration spaces in algebraic topology  is possibly in the Dold-Thom Theorem  in \cite{dt58}.  In this paper the authors consider the \textit{infinite symmetric product} of a based CW-complex $(X,x_0)$ and establish a natural isomorphism from its homotopy groups  onto the
\textit{reduced cellular homology} groups of $(X,x_0)$. More precisely:
\begin{defin}Let $(X,x_0)$ be a based CW-complex. Consider the natural action of the symmetric group $\mathfrak{S}_n$ over $X^n$. The orbit space of this action 
$$SP^n(X)=X^n/\mathfrak{S}_n$$
provided with the quotient topology is called the \textit{n-th symmetric product} of $X$. 
We can include $SP^n(X)$ in $SP^{n+1}(X)$ in the following way
\begin{align*}
SP^n(X)&\rightarrow SP^{n+1}(X)\\
[x_1,\ldots,x_n]&\mapsto[x_0,x_1,\ldots,x_n],
\end{align*}
 Taking colimits over these inclusions we define 
$$SP^\infty(X)=\varinjlim_n SP^n(X),$$
with the colimit topology. $SP^\infty(X)$ is called the infinite symmetric product of $X$.
\end{defin}
\begin{condition}\label{conditionconf}
The topology of $SP^\infty(X)$ as a configuration space is determined by the following properties:
  \begin{enumerate}
    \item If two elements converge to a third element, the label in the limit will be the sum of the labels in the initial 
    points.
    \item If a sequence converges to $x_0$, then the point disappears.
  \end{enumerate}  
\end{condition}
Let $fCW_0$ be the category of based finite CW-complexes and $\Z-ab$ be the category of $\Z$-graded abelian groups. 
  \begin{teor} [Dold-Thom]\label{doldthom}
    There is a natural equivalence between 
    $\pi_*(SP^\infty(-))$ and $\widetilde{H}_*(-)$, where $\widetilde{H}_*(-)$ denotes reduced homology with integer coefficients.
  \end{teor}
 It is an interesting problem to give a description similar to Theorem \ref{doldthom} for generalized homology theories (see for example \cite{shimakawa}). In the current paper we are mainly interested in K-homology. For this case it is possible to assign a configuration space, but we have to consider the \textit{connective} version of K-homology.
  \begin{defin}[Pg. 205 in \cite{ad95}]\label{adamsconnective}
Given a (generalized) reduced homology theory $\mathscr{H}_*$ defined on the category $fCW_0$, one can associate  another homology theory 
$\mathit{h}_*$ such that:
  \begin{enumerate}
    \item There is a natural transformation $$c:\mathit{h}_*\rightarrow\mathscr{H}_*$$such that $c$ is an isomorphism when we evaluate in 
    $(S^0,0)=(\{0,1\}, 0)$ on positive indices.
    \item For every $(X,x_0)\in fCW_0$, and for every $n<0$ we have $\mathit{h}_n(X,x_0)=0$.
  \end{enumerate}
The functor $\mathit{h}_*$ is uniquely determined by these conditions and  is called the \textit{connective} homology theory associated to $\mathscr{H}_*$. In the case of K-homology we denote by $k_*$ the functor \textit{connective K-homology}.
\end{defin}
In \cite{segal1977} Segal  constructs a functor 
$\c(-)$ for connective K-homology analogue to $SP^\infty(-)$.
  \begin{defin}
 Let $(X,x_0)$ be  a pathwise connected, based, finite CW-complex and let $C_0(X)$ be  the C$^*$-algebra of complex-valued continuous functions defined over  $X$ that vanish on the base 
  point $x_0$. Consider the space of $*$-homomorphisms $$\c(X,x_0)=\bigcup_{n\geq0}\Hom^*(C_0(X),\Mat_{n\times n}(\C)),$$with the weak topology, 
  where the 
  union is taken considering the inclusions 
  \begin{align*}\Mat_{n\times n}(\C)&\longrightarrow \Mat_{(n+1)\times(n+1)}(\C)\\
  A&\longmapsto \begin{pmatrix} A && 0\\0 && 0\end{pmatrix}.\end{align*} As $X$ is pathwise connected, $\c(X,x_0)$ is also pathwise connected and therefore its homotopy groups do not depend of the base point $x_0$. We will denote $\c(X,x_0)$ simply by $\c(X)$. We will see below that 
  \begin{equation}\label{nonconnected}\c(X)\simeq\Omega\c(\Sigma X).\end{equation}If $X$ is non pathwise connected we take \ref{nonconnected} as the defintion of $\c(X)$. 
  \end{defin}
  The space $\c(X)$ has a description as a configuration space. If $X$ is connected, elements in $\c(X)$ can be viewed as 
  finite subsets  \begin{displaymath}\{x_1,\ldots,x_n\}\subseteq X-\{x_0\},\end{displaymath} where 
  the $x_i$ are labeled by mutually orthogonal non-zero finite dimensional vector subspace $V_i$ of $\C^\infty$,
  and the topology has the following properties (compare with Condition \ref{conditionconf}).
  \begin{enumerate}
    \item If two points converge to the same points, the label in the limit will be the limit of the direct sum of the labels 
    of the initial points.
    \item If a sequence converges to $x_0$, then the labels converge to 0 (i.e the point disappears).
  \end{enumerate}
A more precise description of these conditions is given in Remark \ref{confdescription}.

Segal obtained a Dold-Thom-theorem for connective K-homology in the following way.
  \begin{teor}[Prop. 1.1 in \cite{segal1977}]\label{main1} Let $(X,x_0)$ be a based finite CW-complex, and denote by $\widetilde{K}$ the reduced K-homology.  There is a natural map 
  \begin{align*}\pi_n(\c(X))\xrightarrow{\ p\ } \widetilde{K}_n(X) \hspace{0.5cm} \text{for }n\geq0\end{align*} such   
  that 
    \begin{enumerate}
      \item This application is an isomorphism when $X=S^0$.
      \item The functor $\pi_*(\c(-))$ is a reduced homology theory in the category of based finite CW-complexes, and the map $p$ is a natural 
      transformation between reduced homology theories.
    \end{enumerate}
  \end{teor}
  By Definition \ref{adamsconnective}, it follows that the functor $\pi_*(\c(-))$ is naturally equivalent to reduced connective K-homology. In this paper, we use ideas of non commutative topology, specifically Kasparov KK-theory, to explain a new proof and a generalization of Theorem \ref{main1} to an equivariant context. More precisely, we have
\begin{teor}Let $G$ be a finite group and $(X,A)$ be a finite $G$-CW-pair. There is a \textit{configuration space} $\c(X/A,G)$, with a continuous $G$-action and a natural map
\begin{align*}\pi_n^G(\c(X/A,G))\longrightarrow {K}_n^G(X,A) \hspace{0.5cm} \text{for }n\geq0\end{align*}
satisfying
  \begin{enumerate}
      \item For every subgroup $H$  of $G$ this map is an isomorphism when $$(X,A)=(G/H,\emptyset)$$ with the natural $G$-action.
      \item The functor  $\pi_*^G(\c(-,G))$ is an equivariant homology theory defined on the category of based finite $G$-CW-pairs.
    \end{enumerate}
\end{teor}
For a precise definition of $\c(X,G)$ see Definition \ref{definconf}.

Now, let us describe the other topic developed in this paper. Segal in \cite{segal1996} 
studies the $\Z\times\Z/2\Z$-graded Hopf algebra
$$\F^q(X)=\bigoplus_{n\geq0}q^n K_{\mathfrak{S_n}}^*(X^n)\otimes\C,$$where $q$ is a formal variable ($q$ gives the grading) and where 
$K^*_G(-)=K^0_G(-)\oplus K^1_G(-)$. The product is defined using the induction functor from 
$(\mathfrak{S}_n\times \mathfrak{S}_m)$-e\-qui\-va\-riant vector bundles to $\mathfrak{S}_{n+m}$-e\-qui\-va\-riant vector bundles, and the coproduct is defined using the corresponding restriction functor. Considering a connected even dimensional $\spinc$-manifold, Segal found a relation between the Hopf algebra 
$\F^q(X)$ and the homology of the configuration space $\c(X)$. Note that $H_*(\c(X);\C)$ is endowed with the Hopf algebra structure induced by the Hopf space structure in $\c(X)$ given by `putting together' the configurations. More precisely Segal proves in \cite{segal1996} the following theorem:
 \begin{teor}\label{main2}Let $M$ be a connected even dimensional $\spinc$-manifold. There is a natural isomorphism of $\Z\times\Z/2\Z$-graded Hopf algebras
  $$H_*(\c(M);\C)\cong\widehat{\mathfrak{F}^q(M)},$$where the completion is taken over the augmentation ideal.
  \end{teor}
 As is proved in \cite{segal1996}, and generalized in \cite{wang2000}, the algebra $ \F^q(X)$ carries several interesting properties. For example it is a free $\lambda$-ring, has the size of a Fock space of
a certain infinite-dimensional Heisenberg superalgebra, and it is the target of the power operations in equivariant K-theory. 

 In \cite{wang2000}, Wang defines an equivariant generalization of $\F^q(X)$. Consider the \textit{wreath} product $G_n=G\wr \S_n$ which is a 
semidirect product of the $n$-th direct product $G^n$ of $G$ and the symmetric group $\S_n$. If $G$ acts on $X$ there is an action 
of the group $G_n$ on $X^n$ induced by the actions of $G^n$ and $\S_n$ on $X^n$.
\begin{defin}Let $X$ be a $G$-CW-complex, we denote by $\F_G^q(X)$ to the $\Z\times\Z/2\Z$-graded Hopf algebra
$$\F_G^q(X)=\bigoplus_{n\geq0}q^nK_{G_n}^*(X^n)\otimes\C.$$
 with the product and coproduct defined using induction and restriction functors on $K_G$. For a precise definition see Definition \ref{indhaces}.
\end{defin}
We generalize Theorem \ref{main2} to an equivariant setting. We prove that the ho\-mo\-lo\-gy groups of the $G$-fixed points of $\c(X,G)$ carry a natural $\Z\times\Z/2\Z$-graded Hopf algebra structure. Using an equivariant version of the Chern character due to L\"uck in \cite{lu2002}, we find a $\Z\times\Z/2\Z$-graded Hopf algebra isomorphism from $H_*(\c(M,G)^G;\C)$  to $\widehat{\F_G^q(M)}$. More precisely we have
\begin{teor}\label{principalhopf}Let $M$ be an even dimensional $G$-$\spinc$-manifold. The $\Z\times\Z/2\Z$-graded complex vector space $H_*(\c(M,G)^G;\C)$ carries a natural $\Z\times\Z/2\Z$-graded Hopf algebra and there is a natural isomorphism of $\Z\times\Z/2\Z$-graded Hopf algebras
  $$H_*(\c(M,G)^G;\C)\cong\widehat{\mathfrak{F}_G^q(M)},$$where the completion is taken over the augmentation ideal.
\end{teor}
Applying this theorem to $M=S^0$, we obtain an expression for the homology of the $G$-fixed points of the equivariant infinite Grasmannian in terms of the representation ring of wreath products. 

This paper is organized as follows. In Section \ref{section1} we recall the definition of the equivariant K-theory spectrum, and fix notations about some functions spaces and $G$-actions over these. We recall  definitions of equivariant homology theories, and its connective versions in the sense of  \cite{lu2002}.

In  Section \ref{section2} we define the configuration space associated to equivariant K-homology and prove that the functor defined as its homotopy groups is a $G$-homology theory. In Section \ref{section2.1}  we define a natural transformation from the homotopy groups of the configuration space to equivariant K-homology and prove that this natural transformation is an isomorphism, when we apply it to  the space $S^0$, with the trivial $G$-action.

In Section \ref{section3} we prove that the functor defined in Section \ref{section2} has an induction structure over finite groups in the sense defined by L\"uck in \cite{lu2002}. We use this fact to prove that the functor is an equivariant homology theory, and finally deduce that it is naturally equivalent to equivariant connective K-homology.

In Section \ref{section4} we describe an isomorphism between the complex homology of the fixed point space of the configuration space and the Hopf algebra $\F^q_G(X)$.
\tableofcontents
\section{Preliminaries}\label{section1}
\subsection{Notation}
Let $(X,x_0)$ and $(Y,y_0)$ be (left) based $G$-spaces. There is a (left) $G$-action on the set of $Map_0(X,Y)$ of based mappings from $X$ to $Y$ defined by 
\begin{align*}
G\times Map_0(X,Y)&\longrightarrow Map_0(X,Y)\\
(g,f)&\longmapsto (x\mapsto g(f(g^{-1}x))).
\end{align*}

If $Map_0(X,Y)$ carries the compact open 
topology and if $X$ is locally compact then the $G$-action  is continuous. Notice that the fixed 
point set $Map_0(X,Y)^G$ consists of the set of $G$-equivariant maps from $X$ to $Y$. The homotopy classes of $Map_0(X,Y)^G$ is denoted by $[X,Y]^G$. A $G$-space $X$ is called $G$-connected if $X^H$ is connected for every $H\subseteq G$.
\subsection{The classifying space for equivariant K-theory } In this section we construct a convenient classifying space for equivariant K-theory for actions of compact Lie groups. This space is the infinite Grassmannian of an infinite dimensional complex representation of $G$ that contains up to isomorphism every irreducible representation of $G$ countably many times. The results in this section are taken from Section XIV.4 in \cite{may91}.
\begin{defin}
  Recall  that  a  \textit{$G$-CW  complex structure}  on  the  pair $(X,A)$  consists  of a  
  filtration of  the $G$-space $X=\bigcup_{-1\leq n } X_{n}$ with $X_{-1}=\emptyset$,  
  $X_{0}=A$ and such that every  space   is inductively  obtained  from  the  previous  one   
  by  attaching  cells  with  pushout  diagrams  
  $$\xymatrix{\coprod_{i} S^{n-1}\times G/H_{i} \ar[r] \ar[d] & X_{n-1} \ar[d] \\ 
  \coprod_{i}D^{n}\times G/H_{i} \ar[r]& X_{n}}.$$
  A $G$-CW-complex $(X,A)$ is called \textit{finite} if it has finitely many cells. Every $G$-CW-complex considered in this paper is assumed to be finite.    
  \end{defin}
\begin{defin}
Let $G$ be a compact Lie group. A complete $G$-universe is a complex separable Hilbert space $\hu_G$ with a linear action of $G$ that contains up to isomorphism every irreducible finite dimensional representation of $G$ infinitely many times. 
\end{defin}
Peter-Weyl theorem gives us a model for a complete $G$-universe. It implies that for $\hu$ a separable complex Hilbert 
space, the Hilbert space $\hu_G=L^2(G)\otimes \hu$ is a complete $G$-universe. 

As in \cite{may91} fixing a complete $G$-universe $\hu_G$, one can construct a representing space $BU_G$ of $G$-equivariant K-theory as a colimit of finite dimensional equivariant Grassmannians. We have the following result.
\begin{teor}\label{grasmaniana}
For a finite based $G$-CW complex $(X,x_0)$, the definition of $\widetilde{K}_G(X)$ as the Grothendieck group of stable isomorphism classes of $G$-vector bundles over $X$ and the classification theorem for complex 
$G$-vector bundles lead to an isomorphism
$$[X, BU_G]^G\cong \widetilde{K}_G (X).$$
\end{teor}
Using Theorem \ref{grasmaniana} we can define the equivariant K-theory spectrum. 
\begin{defin}\label{ktheoryspectrum}
The \textit{equivariant K-theory spectrum} is the sequence of $G$-spaces
$$KU_n=\begin{cases}BU_G &\text{ if $n$ is even,}\\
                    \Omega BU_G &\text{ if $n$ is odd.}\end{cases}$$
\end{defin}
\begin{remark}\label{khomologydefin}Theorem \ref{grasmaniana} implies that if $X$ is a finite based $G$-CW-complex we have isomorphisms 
\begin{align*}
\widetilde{K}_G^n(X)&\cong[X,KU_n]^G, \text{ and }\\
\widetilde{K}^G_n(X)&\cong[X\wedge KU_n]^G \end{align*}for all $n\in\Z$.
\end{remark}
\begin{remark}We are only considering \emph{naive $G$-spectra} because we are only interested to represent equivariant connective K-homology as a $\Z$-graded homology theory.
\end{remark}
\subsection{Finite rank operators space}
For a definition of C$^*$-algebra see \cite{murphy}. We define the C$^*$-algebras  $$\FR_n(\hu_G)=\{A\in \End(\hu_G)\mid \rank(A)\leq n\}$$ 
with the compact 
open topology. In this case, the topology coincides with the weak topology because of the finite rank condition.
We have an inclusion $$\FR_n(\hu_G)\longrightarrow \FR_{n+1}(\hu_G).$$  The group $G$ acts continuously on $\FR_n(\hu_G)$.
\subsection{Equivariant homology theories on $G$-CW-complexes}
We introduce $G$-homology theories and equivariant homology theories. We are following \cite{lu2002}.
\begin{defin}
  A \textit{$G$-homology theory} $\mathscr{H}_*^G$ with values in $R$-modules is a collection of covariant
  functors $\mathscr{H}_n^G$
  from the category of $G$-CW-pairs to the category of $R$-modules, indexed by
  $n\in \Z$, together with natural transformations called the boundary map (here $\H_n^G(A)$ is a shorthand for $\H_n^G(A,\emptyset)$)
  $$\partial_n^G:\mathscr{H}_n^G(X,A)\rightarrow\mathscr{H}_{n-1}^G(A)$$
  for $n\in\Z$, such that the following axioms are satisfied:
    \begin{enumerate}
    \item \textbf{$G$-homotopy invariance.}\\    
    If $f_0$ and $f_1$ are $G$-homotopic maps 
    $(X,A)\rightarrow(Y,B)$ of $G$-CW-pairs, then the induced maps 
    $$f_{0*},f_{1*}:\mathscr{H}_n^G(X,A)\rightarrow \mathscr{H}_n^G(Y,B)$$
    are the same for all $n\in\Z$.
    \item \textbf{Long exact sequence of a pair.}\\
    Given a pair $(X,A)$ of $G$-CW-complexes, there is a long exact sequence
    $$\ldots\xrightarrow{j_*}\mathscr{H}_{n+1}^G(X,A)\xrightarrow{\partial_{n+1}^G}\mathscr{H}_n^G(A)
    \xrightarrow{i_*}\mathscr{H}_n^G(X)\xrightarrow{j_*}\mathscr{H}_n^G(X,A)\xrightarrow{\partial_n^G}\ldots,$$where $i:A\rightarrow X$ and $j:X\rightarrow(X,A)$ are the inclusions.  
    \item \textbf{Excision.}
     
    Let $(X,A)$ be a $G$-CW-pair and let $f:A\rightarrow B$ be a cellular G-map of G-CW-complexes. Equip $(X\cup_f B,B)$ with the induced structure of a 
    $G$-CW-pair. Then the canonical map $F:(X,A)\rightarrow( X\cup_f B,B)$ induces an isomorphism
    $$F_*:\H_n^G(X,A)\rightarrow\H_n^G(X\cup_f B,B).$$
    \item\textbf{Disjoint union axiom.}\\
    Let $\{X_i\mid  i\in I\}$ be a family of $G$-CW-complexes. Denote by $$j_i:X_i\rightarrow \coprod_{i\in I}X_i$$ the
    canonical inclusion. Then the map
    $$\bigoplus_{i\in I}j_{i*}:\bigoplus_{i\in I}\H_n^G(X_i)\rightarrow\H_n^G(\coprod_{i\in I}X_i)$$
    is a group isomorphism.
    \end{enumerate}
  \end{defin}  
Now we will define \textit{equivariant homology theories} following \cite{lu2002}.
\begin{defin}\label{induction}
Let $\alpha: H\rightarrow G$ be a group homomorphism. Given an $H$-space $X$, 
define the 
\textit{induction of $X$ with $\alpha$ }to be the $G$-space $G\times_\alpha X$ which is the quotient of $G\times X$ by the right $H$-action 
\begin{align*}
G\times X\times H&\longrightarrow G\times X\\
((g,x), h)& \longmapsto (g\alpha(h),h^{-1}x).\end{align*} If $\alpha:H\rightarrow G$ is an inclusion, we
also write $G\times_HX$ instead of $G\times_\alpha X$.

An equivariant homology theory $\H_*^?$
with values in $R$-modules consists of a $G$-homology theory $\H^G_*$
with values in $R$-modules for each group $G$ together
with the following so called \textit{induction structure}: given a group homomorphism $\alpha: H\rightarrow G$
and an $H$-CW-pair $(X,A)$ such that $\Ker(\alpha)$ acts freely on $X$, for all $n\in\Z$ there are  natural isomorphisms
$$ind_\alpha:\H_n^H(X,A)\xrightarrow{\ \cong\ }\H_n^G(G\times_\alpha X,G\times_\alpha A)$$
satisfying:
  \begin{enumerate}
  \item \textbf{Compatibility with the boundary homomorphisms.}
  $$\partial_n^G\circ ind_\alpha=ind_\alpha\circ\partial_n^H.$$
  \item \textbf{Functoriality.}
  Let $\beta:G\rightarrow K$ be another group homomorphism such that $\Ker(\beta\circ\alpha)$ acts freely on
  $X$. Then we have for $n\in\Z$
  $$ind_{\beta\circ\alpha}=f_{1*}\circ ind_\beta\circ ind_\alpha:\H_n^H(X,A)\longrightarrow\H_n^K(K\times_{\beta\circ\alpha}(X,A)),$$
  where 
  \begin{align*}
  f_1:K\times_\beta(G\times_\alpha X,G\times_\alpha A)&\xrightarrow{\ \cong\ }(K_{\beta\circ\alpha}X,K_{\beta\circ\alpha}A)\\ 
  (k,g,x)&\longmapsto(k\beta(g),x)\end{align*} is the natural $K$-homeomorphism.
  \item \textbf{Compatibility with conjugation. }We denote by $c(g):G\rightarrow G$ the conjugation homomorphism $c(g)(h)=ghg^{-1}$. For $n\in\Z$, $g\in G$ and a $G$-CW-pair $(X,A)$ the homomorphism 
  $$ind_{c(g)}:\H_n^G(X,A)\longrightarrow\H_n^G(G\times_{c(g)}X,G\times_{c(g)}A)$$ agrees with $f_{2*}$ for the $G$-homeomorphism 
  \begin{align*}
  f_2:(X,A)&\longrightarrow (G\times_{c(g)}X,G\times_{c(g)}A)\\
  x &\longmapsto (1,g^{-1}x).\end{align*} 
  \end{enumerate}
\end{defin}  
For examples of equivariant homology theories see examples 1.3, 1.4 and 1.5 in \cite{lu2002}.

\begin{remark}Notice that the notion of equivariant homology theory introduced here is more restrictive than the notion of \emph{compatible family of equivariant ho\-mo\-lo\-gy theories} in \cite{lm86} because that notion only consider the case when the homomorphism $\alpha$ is an inclusion.
\end{remark}

It is natural (in analogy with \cite[Pg. 205]{ad95}) to associate to the equivariant K-homology a \textit{connective} 
$G$-homology theory in the following way.
  \begin{prop}\cite[Secc. 3]{greenlees2005}\label{connectiveversion}
There is a $G$-homology theory $\mathit{k}_*^G$ such that:
  \begin{enumerate}
  \item There is a natural transformation $c:\mathit{k}_*^G\rightarrow K_*^G$ such that $c$ is an isomorphism when we evaluate in homogeneous spaces
  $G/H$ over positive indexes.
  \item For every $G$-CW-complex $X$ and for every $i<0$, $\mathit{k}_i^G(X)=0$.
  \end{enumerate}
The functor $\mathit{k}_*^G$ is uniquely determined up to natural equivalence by these conditions and is called \textit{$G$-equivariant connective K-homology}.
\end{prop}

\section{The configuration space}\label{section2}
Let us recall from Section \ref{section0} that in the non-equivariant case, the connective K-homology groups of a finite based CW-complex 
   $(X,x_0)$ are naturally isomorphic to the homotopy groups of the configuration space $\c(X)$ of 
   finite subsets of $X-\{x_0\}$, where each point is labelled by  
  an element of  the infinite Grassmannian of $\C^\infty$. It would be natural to describe the equivariant K-homology of a finite based $G$-CW-complex $(X,x_0)$ as the equivariant homotopy groups of a configuration  space $\c(X,x_0,G)$ whose elements are finite subsets  of $X-\{x_0\}$ and  the labels are elements of an appropriate equivariant analogue to the infinite Grassmannian. In this section, we construct an equivariant analogue of $\c(X)$ and prove that its equivariant homotopy groups form a  $G$-homology theory.
\subsection{Descriptions of the configuration space}In this section we define the $G$-space $\c(X,G)$ in terms of spaces of $\ast$-homomorphisms. We also give a description of $\c(X,G)$ as a configuration space and a geometric description of the $G$-action on it.
  \begin{defin}\label{definconf}Let $G$ be a finite group and $(X,x_0)$ be a based $G$-connected, $G$-CW-complex.
  Let $\c(X, x_0,G)$ be the \textit{$G$-space of configurations of complex vector spaces over $(X,x_0)$}, defined as the following union, with respect to the inclusions $\FR_n(\hu_G)\rightarrow\FR_{n+1}(\hu_G)$
$$\c(X,x_0,G)=\bigcup_{n\geq0}\Hom^*(C_0(X),\FR_n(\hu_G)),$$with the compact open topolo\-gy. 
  Notice that * refers to *-homomorphism. 
  
  We endow to $\c(X,x_0,G)$ 
with a $G$-action in the following way. If $F\in \c(X,x_0,G)$, we define 
\begin{align*}g\cdot F: C_0(X)&\longrightarrow \FR_n(\hu_G)\\
f&\longmapsto g\cdot F(g^{-1}\cdot f).\end{align*}
 This action is continuous. As $X$ is $G$-connected, $\c(X,x_0,G)$ is also $G$-connected and its equivariant homotopy groups do not depend on the base point $x_0$, we can denote $\c(X,x_0,G)$ simply by $\c(X,G)$. 
  \end{defin}
  
The space $\c(X,G)$ has a description as a configuration space. To obtain that description we need to recall the Gelfand-Naimark theorem, for a proof of the Gelfand-Naimark theorem see \cite[Thm. I.3.1]{algbyexample}.
\begin{defin}The \textit{spectrum} of the C$^*$-algebra $C_0(X)$ is the based topological space of characters $$\widehat{C_0(X)}=\Hom^*(C_0(X),\C),$$
with the strong$^*$-topology and with base point the zero character ${\bf0}$.
\end{defin}

\begin{teor}[Gelfand-Naimark]Evaluation gives us a homeomorphism of based spaces
\begin{align*}(X,x_0)\xrightarrow{\ \xi\ }&\ (\widehat{C_0(X)},{\bf0})\\
x\longmapsto&\  \xi(x)[f]=f(x). \end{align*}
\end{teor}
  Let $F\in Hom^*(C_0(X),\FR_n(\hu_G))$. For every $f\in C_0(X)$ we have
  \begin{align*}
  F(f)(F(f))^*&=F(ff^*)\\
  &=F(f^*f)\\&=F(f)^*F(f).
  \end{align*}  
  Then the operator $F(f)$ is normal, so it is 
diagonalizable. Moreover as $C_0(X)$ is commutative the corresponding eigenspaces do not depend on $f$ because all elements in $$\{F(f)\mid f\in C_0(X)\}\subseteq FR_n(\hu_G)$$ are simultaneously diagonalizable. 

Taking eigenvalues give us a continuous map
$$\c(X,G)\longrightarrow SP^\infty(\widehat{C_0(X)},{\bf0})$$
composing with the Gelfand-Naimark homeomorphism we have a map
$$\c(X,G)\longrightarrow SP^\infty(X,x_0).$$ In the same way we have the following description of $\c(X,G)$.
\begin{remark}\label{confdescription}
The space $\c(X,G)$ has a description as a configuration space.
  \begin{enumerate}
  \item For every element $F\in\c(X,G)$ consider the set of eigenvalues of $F$, 
  $$\sum_{i}m_ix_i\in SP^\infty(X,x_0)$$with $x_i\neq x_j$ if $i\neq j$. We can associate a configuration 
  $$\{(x_1,V_1),\ldots, (x_n,V_n)\},$$
 where the corresponding $V_i\subseteq\hu_G$ is the eigenspace corresponding to $x_i$; it is a finite dimensional vector subspace of $\hu_G$ such that if $x_i\neq x_j$ then $V_i\perp V_j$.
  \item The topology of $\c(X,G)$ can be recovered in the above description in the following way.
  Let $(F_i)_{i\geq0}$ be a sequence converging to $F$ in $\c(X,G)$. Suppose that each $F_i$ is characterized by a configuration
  $$\{(x_1^i,V_1^i),\ldots,(x_{n_i}^i,V_{n_i}^i)\},$$
  and $F$ by
  $$\{(x_1,V_1),\ldots, (x_n,V_n)\}.$$
  Then the set $\{x_1^i,\ldots,x_n^i\}$ viewed as an element in $SP^\infty(X,x_0)$ converges to $\{x_1,\ldots,x_n\}$. That means that up to a reordering of the indexes any $x_k^i$ converges to a unique $x_l$, and then on the labels one should impose the condition
  $$\bigoplus_{\{k\mid x_i^k\rightarrow x_l\}}V_k^i\rightarrow V_l$$
  as elements in $BU_G$. 
  \end{enumerate}
  \end{remark}  
 
\begin{remark}\label{conf} In the context of Remark \ref{confdescription} we can describe the $G$-action as follows. For $F\in \c(X,G)$ 
represented by a configuration 
$$\{(x_1,V_1),\ldots,(x_n,V_n)\},$$ 
we have $$(g\cdot F)(f)(gv)=g[(F(g^{-1}f))(g^{-1}gv)]=g[(f(gx_i))(v)]=(f(gx_i))(gv),$$ 
where the last equality follows because $f(gx_i)$ is a scalar. Then one has a natural continuous  $G$-action on $\c(X,G)$ on the description given in Remark \ref{confdescription} 
 $$g\cdot\{(x_1,V_1),\ldots,(x_n,V_n)\}=\{(gx_1,gV_{x_1}),\ldots,(gx_n,gV_{x_n})\}.$$
\end{remark}
\subsection{The homotopy groups of $\c(X,G)$}
We denote by $CW^{(2)}_G$ the category of pairs $(X,A)$ where $X$ is a $G$-CW-complex and $A$ a closed $G$-ANR, which means that there exists a $G$-open set $U\supseteq A$ such that 
$U$ is a weak $G$-
deformation retract of $A$. In \cite{mit83} is proved that $CW^{(2)}_G$  is naturally equivalent to the category of $G$-CW-pairs.
\begin{defin}
We define a family of covariant functors from $CW^{(2)}_G$ to the category of $\Z$-graded abelian groups. Let $(X,A)\in CW^{(2)}_G$ and suppose that $X$ is $G$-connected. Note that $(X/A, A/A)$ is a based $G$-space. We define the functors
\begin{align*}
\underline{k}&_*^G(-,-):CW^{(2)}_G\longrightarrow \Z-Ab\\&\underline{k}_n^G(X,A)=\pi_n(\c(X/A,G)^G).\end{align*}

If $X$ is non-$G$-connected we extend the functor $\underline{k}_*^G$  by defining 
 $$\underline{k}_*^G(X,A)=\pi_{*+1}^G(\c(\Sigma (X/A),G)).$$
 
Let us define the unreduced functor. Given a $G$-space $X$ consider the space $X_+=X\cup\{+\}$ with base point of  $+$, $G$ acts trivially on $+$. Define  $$\underline{k}_*^G(X)=\underline{k}_*^G(X_{+},+).$$
\end{defin}
\begin{teor}\label{homologytheory}The functor $\underline{k}_*^G$ is a $G$-homology theory.
\end{teor}
We will prove this theorem in the next Section.
\subsection{Proof of Theorem \ref{homologytheory}}
First we prove that $\underline{k}_*^G$ satisfies the homotopy axiom, the excision axiom and the disjoint union axiom. For the long exact sequence axiom  we will need some lemmas.
\begin{prop} The functor $\underline{k}_*^G$ satisfies the homotopy axiom.
 \end{prop}
\begin{proof} 
We can define a map 
\begin{align*}\omega:I\times \c(X,G)&\to \c(I\times X,G)\\
(t,F)&\mapsto \left(f\mapsto F(f(t,-))\right).\end{align*}
Where $f\in C_0(I\times X)$. On the other hand if $H:I\times X\to Y$ is a homotopy 
we have the induced map $H_*:\c(I\times X,G)\to \c(Y)$. The map $H_*\circ\omega$ gives us a homotopy. Taking the induced map on homotopy groups we obtain that the induced maps by $H(0,-)$ and $H(1,-)$ in $\underline{k}_*^G$ are the same.

\end{proof}
\begin{prop}
The functor $\underline{k}_*^G$ satisfies the excision axiom.
\end{prop}
\begin{proof}
Given a $G$-map $f:A\to B$, there is a canonical $G$-homotopy equivalence $\bar{F}:X/A\rightarrow X\cup_fB/B$. Applying the homotopy axiom to $\bar{F}$ the result follows as stated.
\end{proof}
\begin{prop}The functor $\underline{k}_*^G$ satisfies the disjoint union axiom.
\end{prop}
\begin{proof}
  Let $X=\coprod_{i\in I}X_i$ be a disjoint union, in this case 
  $$\c(X,G)=\bigcup_{n\geq0}\Hom^*(C_0(\Sigma(\coprod_{i\in I}X_i)),\FR_n(\hu_G))^G.$$
  First notice that $C_0(\Sigma(\coprod_{i\in I}X_i))$ can be identified with $\prod_{i\in I}C_0(\Sigma X_i)$. For $j\in I$ we denote the inclusion by $$\iota_j:\Sigma X_j\longrightarrow\coprod_{i\in I}\Sigma X_i,$$
  taking pullback we have
  $$\iota_j^*:C_0(\coprod_{i\in I}\Sigma X_i)\longrightarrow C_0(\Sigma X_j).$$
   
  Let $(F_i)_{i\in I}\in \prod_{i\in I}\Hom^*(C_0(\Sigma X_i),\FR_n(\hu_G))$ be an element in the \emph{weak} product, the union of the products of finitely many factors. We can define $$F\in \Hom^*(\prod_{i\in I}C_0(\Sigma X_i),\FR_n(\hu_G))$$ by
  \begin{equation}\label{disjoint}F((f_i)_{i\in I})=\sum_{j\in I}F_j(\iota_j^*((f_i)_{i\in I})),\end{equation}
  as the rank of $F$ is finite, the sum on the right side is finite for every element in $\prod_{i\in I}C_0(\Sigma X_j)$. Equation \ref{disjoint} gives us a $G$-homeomorphism
  $$\prod_{i\in I}\Hom^*(C_0(\Sigma X_i),\FR_n(\hu_G))\longrightarrow\Hom^*(\prod_{i\in I}C_0(\Sigma X_j),\FR_n(\hu_G)) $$
  Where in the left side we are taking the weak product.
  Taking union and homotopy groups, for every $k\geq0$, we have an isomorphism  
  $$\bigoplus_{i\in I}\pi_k(\c(X_i,G))\xrightarrow{\ \cong \ }\pi_k(\c(\coprod_{i\in I}X_i,G)^G).$$We conclude that $\underline{k}_*^G$ satisfies the disjoint union axiom.
\end{proof}
To prove the long exact sequence axiom for $\underline{k}^G_*$, we need to recall the definition of $G$-quasifibration.
\begin{defin}[\cite{wa80}]
  A map $p:E\rightarrow B$ on the category of based $G$-CW-complexes is a \emph{$G$-quasifibration} if  for every $b\in B$, $x_0\in p^{-1}(b)$ and $H\subseteq G_b=\{g\in G\mid gb=b \}$, the
induced map $$p_\star : \pi_i(E^H, p^{-1}(b), x_0 )\to\pi_i(B^H, b)$$ is an isomorphism for all $i\geq0$.
 \end{defin}

The proof of a fibration inducing a long exact sequence on homotopy groups only uses the weaker condition of quasifibration, and the next proposition follows.
\begin{prop}\label{lesquasi}
If $p:E\rightarrow B$ is a $G$-quasifibration, then for every $b\in B$ and $H\subseteq G_b$ there is a long exact sequence of homotopy groups
$$\cdots\longrightarrow\pi_n(E^H,p^{-1}(b),x_0)\xrightarrow{\ p_*\ }\pi_n(B^H,b)\xrightarrow{\ \partial\ }\pi_{n-1}(p^{-1}(b),x_0)\longrightarrow\cdots.$$
For every $x_0\in p^{-1}(b)$.
\end{prop}
We need to recall  the following lemma.
  \begin{lema}\label{edold}(\cite{wa80})
  A map $p : E\rightarrow B$ is a $G$-quasifibration if any one of the following conditions
  is satisfied:
    \begin{enumerate}
      \item \label{edold1}The space $B$ can be decomposed as the union of $G$-open sets $V_1$ and $V_2$ such that each of
      the restrictions $p^{-1}(V_1)\rightarrow V_1$, $p^{-1}(V_2)\rightarrow V_2$, and $p^{-1}(V_1\cap V_2)\rightarrow V_1 \cap V_2$ 
      are $G$-quasifibrations.
      \item \label{edold2}The space $B$ is the union of an increasing sequence of $G$-subspaces $B_1\subseteq B_2\subseteq\cdots$ with the property 
      that each $G$-compact set in $B$ lies in some $B_n$, and such that each restriction $p^{-1}(B_n)\rightarrow B_n$ is a 
      $G$-quasifibration.
      \item \label{edold3}There is a $G$-deformation $\Gamma_t$ of $E$ into a $G$-subspace $E_0$, covering a deformation $\bar{\Gamma}_t$ of $B$ 
      into a $G$-subspace $B_0$, such that the restriction $E_0\rightarrow B_0$ is a $G$-quasifibration and 
       $\Gamma_1 :p^{-1}(b)\rightarrow p^{-1}(\bar{\Gamma}_1(b))$ is a $G_b$-weak 
      homotopy equivalence for each $b\in B$.
    \end{enumerate}
  \end{lema}
For $(X,A)\in CW^{(2)}_G$ there is a canonical  inclusion $i_*:\c(A,G)\to\c(X,G)$, induced by $i:A\to X$ 
and a canonical projection  $p_*:\c(X,G)\rightarrow \c(X/A,G)$ induced by $p:X\to X/A$. For simplicity we identify the C$^*$-algebra $C_0(X/A)$ with 
$$C_0(X,A)=\{f:X\rightarrow\C\text{ continuous }\mid f(A)=\{0\}\}\subseteq C_0(X).$$We can describe $p_*$ using this identification. For 
$f\in C_0(X/A)$ and $F\in\c(X,G)$ note that $$p_*(F)(f)=F(f)$$ 

Let $N$ be a $G$-neighborhood of $A$ in $X$, such that $N$ is a $G$-deformation retract of $A$; 
we denote the $G$-retraction
 by $r:N\rightarrow A$. The map $r$ induces a $G$-map $r^*:C_0(A)\rightarrow C_0(N)$.
 
The pair of $G$-open sets $\{N,X- A\}$ is a $G$-open covering of $X$, and as a consequence there is a $G$-equivariant  partition of unity 
$\{\rho_1,\rho_2\}$ (it can be obtained simply by taking the non-equivariant partition of unity and averaging by $G$). The partition is associated to the covering in a way that $supp(\rho_1)\subseteq N$ and $supp(\rho_2)\subseteq X- A$. 

We introduce our first technical lemma that we will use together with lemma \ref{edold}(3).
  \begin{lema}\label{equiv}For every $b\in\c(X/A,G)$ there exists a map
  $$\mu_b:p_*^{-1}(b)\longrightarrow\c(A,G)$$
  which is a $G_b$-homotopy equivalence.
  \end{lema}
  \begin{proof}
  If $F\in p^{-1}_*(b)$ where $F:C_0(X)\longrightarrow FR_n(\hu_G) $ using the identification $C_0(X/A)\cong C_0(X,A)$ we have that $$F\mid {C_0(X,A)}=b.$$ 
  
  Define for $F\in p^{-1}(b)$ and $f\in C_0(A)$
  \begin{align*}\mu_b:p^{-1}(b)&\longrightarrow\c(A,G)\\
   F&\longmapsto  \mu_b(F)(f)=F(\rho_1.r^*(f)).\end{align*}
  
  Now define a homotopy inverse of $\mu_b$. Let $$\gamma_b:\c(A,G)\longrightarrow p^{-1}(b)$$be a map defined for $f\in C_0(X)$ and $F\in\c(A,G)$, by
  $$\gamma_b(F)(f)=F(f\mid_A)+b(\rho_2.f).$$
  
  The composition $\mu_b\circ\gamma_b:\c(A,G)\longrightarrow\c(A,G)$, is $G_b$-homotopic to the identity, because we have for 
  $F\in\c(A,G)$ and $f\in C_0(A)$:
\begin{align*}
      (\mu_b\circ\gamma_b)(F)(f)&=\gamma_b(F)(\rho_1.r^*(f))\\
                      &=F((\rho_1.r^*(f))\mid_A)+b(\rho_2\rho_1r^*(f))\\
                      &=F(f)+b(\rho_2\rho_1r^*(f)).
                                                     \end{align*}
Choosing a $G_b$-equivariant  path between $b$ and $0$ 
  define a deformation $H_t$ from the last expression  to the identity as follows. Let $\gamma_t$ be a path that connects $b$ and $0$. Consider the map
  $$H_t:\c(A,G)\longrightarrow\c(A,G)$$$$H_t(F)(f)=F(f)+\gamma_t(\varphi^{-1}(\rho_2\rho_1r^*(f)))$$The map $H_t$ is the desired homotopy.
  
  The composition $\gamma_b\circ\mu_b:p_*^{-1}(b)\rightarrow p_*^{-1}(b)$ is $G_b$-homotopic to the identity because for 
$F\in p_*^{-1}(b)$ and $f\in C_0(X)$ we have
  \begin{align*}
  \gamma_b\circ\mu_b(F)(f)&=\mu_b(F)(f\mid_A)+b(\varphi^{-1}(\rho_2f))\\
                                             &=F(\rho_1r^*(f\mid_A))+b(\varphi^{-1}(\rho_2f))\\
                                             &=F(\rho_1r^*(f\mid_A)+\rho_2f),\end{align*}
  where the last equality follows because  $\rho_2f\in C_0(X,A)$,  but $\rho_1r^*(f\mid_A)+f\rho_2$ can be continuously deformed
  (in $C_0(X)$) to $\rho_1f+\rho_2f=f$ by a linear homotopy. Note that since $\rho_1$, $\rho_2$  and $r^*$ are $G_b$- equivariant maps, 
  then $\mu_b$ is a $G_b$-homotopy equivalence.
    \end{proof}
\begin{teor}\label{quasifib1}The map 
 $$p_*:\c(X,G)\longrightarrow \c(X/A,G)$$ is a $G$-quasifibration.
\end{teor}
 \begin{proof}

 Let us filter $\c(X/A,G)$ by $G$-closed spaces in the following way $$\c^n(X/A,G)=\{F\in\ca\mid \mbox{rank}(F)\leq n\}.$$

 We want to prove in the following lines that $p\mid p^{-1}(\cn)$ is a quasifibration and using Lemma \ref{edold}(2) conclude that $p$ is a quasifibration. We proceed by induction.

  \begin{lema}\label{filtration1}
   The restriction map $p\mid p^{-1}(\cn)$ is a $G$-quasifibration.
  \end{lema}
  \begin{proof}
The case   $n=0$ is trivial because it is a map from a point to a point. Now suppose that $p\mid p^{-1}(\cn)$ is a quasifibration. We will prove that $p\mid p^{-1}(\c^{n+1}(X/A,G))$ is a 
   quasifibration in two steps.\\

  \textbf{Step 1:} Let us show that we can find a $G$-open set $U$ in $p^{-1}(\c^{n+1}(X/A,G))$ such that $U$ 
  is a (weak) deformation retract of $p^{-1}(\cn)$ and $p\mid p^{-1}(U)$ 
  is a quasifibration. The existence of $U$ will be proved in the following argument.   Recall that there is a $G$-neighborhood   $N\subseteq X$ of $A$ such that $N$ is a $G$-deformation retract of $A$. 
 
  Let $r_t:X\rightarrow X$, ($t\in[0,1]$) be a homotopy, such that $r_1(N)=A$, $r_t(a)=a$ for every $a\in A$, and $r_0=id_X$. The homotopy $r_t$ induces also a homotopy $$\bar{r}_t: X/A\to X/A.$$ Consider 
  the map  $$(\bar{r}_1)_*:\cnn\longrightarrow\cnn.$$ 
  
  The set $(\bar{r}_1)^{-1}_*(\cn)$ is a closed subset of $\cnn$. The inclusion $i^*:C_0(X,N)\to C_0(X,A)$ induces a map $$i_*:\c^{n+1}(X/A,G)\longrightarrow\c^{n+1}(X/N,G).$$ Define the set 
  $$W=i_*^{-1}(\c^n(X/N,G))\subseteq \c^{n+1}(X/A,G).$$
  It is an open set in $\cnn$ such that $$(\bar{r}_1)_*^{-1}\left(\cn\right)\supseteq W\supseteq\cn$$ and $W$ deforms in $\cn$. Consider the set $U=p^{-1}(W)$. The homotopy 
  $$r_{t*}:p^{-1}(\cnn)\longrightarrow p^{-1}(\cnn)$$restricted to $U$ is a weak deformation retract of $U$ to $p^{-1}(\cn)$, and the homotopy 
  $r_{t*}$ covers $\bar{r}_{t*}$. To conclude that $p:U\rightarrow p(U)$ is a $G$-quasifibration it is enough 
  to prove that $r_{1*}:p^{-1}(b)\rightarrow p^{-1}(\bar{r}_{1*}(b))$ is a weak $G_b$-homotopy equivalence and then  use Lemma \ref{edold}(3). The proof of $\bar{r}_{1*}$ is a $G_b$-homotopy equivalence is completely  analogue to the proof 
  of Lemma \ref{equiv}.

  \textbf{Step 2:} Prove that $$p\mid(p^{-1}(\can))$$and$$p\mid(p^{-1}(\can)\cap p^{-1}(U))$$ are $G$-quasifibrations.

  In this step we prove directly that the induced maps in the corresponding homotopy groups are isomorphisms. We only prove that $$p\mid(p^{-1}(\can))$$ is a $G$-quasifibration. The case for $$p\mid(p^{-1}(\can)\cap p^{-1}(U))$$ is completely analogue.
  
  In order to prove that the induced map on homotopy groups is surjective we want to define a continuous section $$\hspace{1.5cm}s:\can\longrightarrow p^{-1}(\can).$$
  As in Remark \ref{confdescription} every $F\in\can$ can be identified with a configuration   
  $$\{(x_1,V_1),\ldots (x_n,V_n)\}.$$ There exists $s<1$ such that for every $t<1$ with 
  $s\leq t$, and for every $f\in C_0(X)$ $$F((\rho_1\circ r_s)\cdot f)=F((\rho_1\circ r_t)\cdot f).$$ 
  Define $$s(F)(f)=F((\rho_1\circ r_s)\cdot f)$$
  and since $s(F)$ is defined with multiplication by a continuous map, and $s(F)$ is continuous for every $F$, then $s$ is well defined. To see that $s$ is continuous consider a convergent sequence in $\can$ 
   $$(F_k)_k\rightarrow F\in \can.$$ Each $F_k$ has eigenvalues $x_i^k$, each sequence $(x_i^k)$ cannot converge to 
  $x_0$ because this implies that $F\in\cn$, and therefore there is a $t\in[0,1)$ such that for every $k$ and for $f\in C_0(X)$ 
  $$s(F_k)(f)=F_k((\rho_1\circ r_s).f)\text{ and } s(F)(f)=F((\rho_1\circ r_s).f),$$where the result of $s$ applied to this sequence is obtained
  by multiplying 
  by a continuous function, and hence $s(F_k)$ converges to $F$, i.e the map $s$ is continuous. 
  
  We have that $p\circ i= id_{\can}$, and as a consequence 
  \begin{multline*}  
  \hspace{1cm}p_*:\pi_i^{G_b}(p^{-1}(\can),p^{-1}(b))\longrightarrow\\ \pi_i^{G_b}(\can,b)
  \end{multline*}
  is surjective. To see that $p_*$ is injective, let $$g:(D^i,\partial D^i)\rightarrow p^{-1}(\can)$$be a representative element of 
  $\Ker(p_*)$. Therefore $p\circ g\simeq_{G_b} b$.

  Let $$\gamma:(D^i,\partial D^i)\times I\rightarrow \can$$ be a map such that $\gamma(-,0)=p\circ g$ and $\gamma(-,1)=b$. We define 
  \begin{align*}
  \widetilde{\gamma}:(D^i,\partial D^i)\times I&\longrightarrow (p^{-1}(\can),p^{-1}(b))\\
  (a,t)&\longmapsto\widetilde{\gamma}(a,t)(f)=\gamma(a,t)(\rho_1.f)+g(a)(\rho_2.f) .\end{align*}This is a homotopy that starts in $g$ and ends in an element of 
  $p^{-1}(b)$, so
  $[g]=0$ in $\pi_i^{G_b}(p^{-1}(\can),p^{-1}(b))$, hence the kernel is trivial.
  
  Then using lemma \ref{edold}(1) we conclude that $p\mid p^{-1}(\c^{n+1}(X/A,G))$ is a $G$-quasifibration.
  \end{proof}
  Using Lemma \ref{edold}(2) together with Lemma \ref{filtration1} we conclude that the map 
 $$p_*:\c(X,G)\longrightarrow \c(X/A,G)$$ is a $G$-quasifibration, it proves Theorem \ref{quasifib1}. Finally using Proposition \ref{lesquasi}, we have proved that $\underline{k}_*^G$ satifies the long exact sequence axiom, and then $\underline{k}_*^G$ is a $G$-homology theory.
\end{proof}
\section{Equivariant connective K-homology}\label{section2.1}
So far we have proved that the functor $\underline{k}_*^G$ is a $G$-homology theory for every finite group $G$. Now we will define a natural transformation $\mathfrak{A}$ from $\underline{k}_*^G$ to equivariant K-homology
 such that the map
   $$\underline{k}^G_n(G/H)=[S^n,\c(G/H,G)]^G \xrightarrow{\ \mathfrak{A}(G/H)_n\ }K_n^G(G/H)\text{ for }n\geq0$$
   is an isomorphism. Let us start with some preliminaries.
  \subsection{Equivariant KK-theory}Atiyah proved that elliptic operators between sections of two vector bundles $E\to X$ and $F\to Y$ give rise to maps between K-theory groups of $X$ and $Y$ (see for example \cite{atiyahbottper}). Kasparov extend this idea to a \textit{generalized elliptic operator} (see Definition \ref{generalizedelliptic}). Given two C$^*$-algebras $C$ and $B$, a generalized elliptic operator $\eta$ defined between Hilbert modules over $C$ and $B$ induces a map in K-theory
  $$K_*(B)\xrightarrow{\ -\sharp\eta\ }K_*(C).$$ 
 The study of the homotopy classes of generalized elliptic operators (see Definition \ref{homotopyKK}) is very important in non commutative topology and in index theory (for a good introduction to the subject see for example \cite{higsonprimer}). We will use the properties of this homotopy classes, in particular a convenient form of Bott periodicity in order to construct the natural transformation $\mathfrak{A}$. In this section we follow \cite[Chapter VIII]{bl98}.
  \begin{defin}A $\Z/2\Z$-graded $G$-C$^\ast$-algebra $C=C^0\oplus C^1$, is a C$^\ast$-algebra equipped with an action of $G$ by $^\ast$-automorphisms preserving the grading. Given two $\Z/2\Z$-graded $G$-C$^*$-algebras $C$ and $B$, a map $\phi:C\to B$ is called a $G$-*-homomorphism if it is $G$-equivariant and a *-homomorphism.
  \end{defin}
 Given a $\Z/2\Z$-graded $G$-C$^\ast$-algebra $B$, a $\Z/2\Z$-graded Hilbert $B$-module is a $\Z/2\Z$-graded $B$-module with a $\Z/2\Z$-graded $B$-valued inner product. Over a $\Z/2\Z$-graded Hilbert $B$-module $E$ one can define the $\Z/2\Z$-graded $G$-C$^\ast$-algebra of \emph{bounded operators} denoted by $\mathfrak{B}(E)$ and the $\Z/2\Z$-graded $G$-C$^\ast$-algebra of \emph{compact operators} denoted by $\mathfrak{K}(E)$, in both cases we take the usual grading coming from $E$. For precise definitions see \cite[Secc. VI.13]{bl98}.
 
  \begin{defin}\label{generalizedelliptic}
 Let $C$ and $B$ be ($\Z/2$-) graded $G$-C$^\ast$-algebras. We denote by $\E_G(C,B)$ to the set
  of \textit{Kasparov $G$-modules} (or generalized $G$-elliptic operators) for $(C, B)$, that is the set of triples $(E,\phi,F)$ such that
  \begin{enumerate}
  \item $E$ is  a graded countably generated Hilbert $B$-module with a continuous $G$-action.
  \item $\phi:C\rightarrow \mathfrak{B}(E)$ is a graded *-
  homomorphism. 
  \item $F$ is a $G$-continuous
  operator in $\mathfrak{B}(E)$ of degree 1, 
  such that for every $c \in C$ and $g\in G$ 
  \begin{enumerate}
  \item $F\phi(c)-\phi(c)F$, \item $(F^2 - Id)\phi(c)$, \item $(F - F^*)\phi(c)$ and \item $(g\cdot F-F)\phi(c)$\end{enumerate}
  are all in $\mathfrak{K}(E)$. 
\end{enumerate}  
  The set $\D_G(C, B)$  of degenerate Kasparov modules is the set of triples in $\E_G(C,B)$ for which 
  \begin{enumerate}
  \item $F\phi(c)-\phi(c)F=0$, \item $(F^2 - Id)\phi(c)=0$, \item $(F - F^*)\phi(c)=0$, and \item $(g\cdot F-F)\phi(c)=0$,\end{enumerate} for all $c\in C$ and $g\in G$.
  \end{defin}
  \begin{ejem}\label{fundamental}
  Let $\phi: C\rightarrow \mathfrak{K}(\hu_G)$ be a graded $G$-*-homomorphism. Then
  $(\hu_G,\phi,0)$ is a Kasparov $G$-$(C,\C)$-module.
  \end{ejem}
  \begin{defin}\label{homotopyKK}Let $IB=C([0,1],B)$ be the C${ }^*$-algebra of continuous maps from $[0,1]$ to $B$. A \textit{homotopy} connecting $(E_0,\phi_0,F_0)$ and $(E_1,\phi_1,F_1)$ in $\E_G(C,B)$ is an
  element $(E,\phi,F)$ of $\E_G(C,IB)$ for which $(E\widehat{\otimes}_{f_i}B,f_i\circ\phi,f_{i*}(F))$ is $G$-unitary equivalent to $(E_i,\phi_i,F_i)$,
  where $f_i$, for $i = 0$, $1$, is the evaluation homomorphism from $IB$ to $B$.
  \end{defin} The notion of homotopy
  respects direct sums. Homotopy equivalence is denoted by $\sim_h$. If we have $E_0 = E_1$, then a
  \textit{standard homotopy} is a homotopy of the form $E = C([0,1],E_0)$ (which is a
  Hilbert $IB$-module in the obvious way), with  $\phi =(\phi_t)$, and $F = (F_t)$, where $t\rightarrow F_t$
  and $t\rightarrow \phi_t(c)$ are strong $G$-*-operator continuous for each $c$. 
  A standard homotopy where in addition $\phi_t$ is constant
  and $F_t$ is norm-continuous is called an \textit{operator homotopy}.
  \begin{defin} Direct sums turns $\E_G(C,B)$ into an abelian semigroup. We denote by $KK_G(C,B) $ to the set of equivalence classes
  of $E_G(C,B)$ under $\sim_h$. More generally, we set $$KK^n_G(C,B) = KK_G(C, B\otimes \C\text{liff}(n))$$ where $G$ acts trivially in 
  $\C\text{liff}(n)$.  In each case, the set is an abelian semigroup under direct sum.
  \end{defin}
  The bifunctor $KK_G^n(-,-)$ gives us a convenient description of equivariant K-homology that we will use to define the natural transformation.  We have the following result.
  \begin{prop}[Corol. 18.5.4 in \cite{bl98}]\label{KK}
  There are natural isomorphisms 
  \begin{align*}
  \widetilde{K}^n_G(X)& \cong KK_G^n(\C,C_0(X))\text{ and }\\   
  \widetilde{K}^G_n(X)& \cong KK_G^n(C_0(X),\C).\end{align*}
  \end{prop}
  We need also a result about the invariance of KK-theory by compact perturbations.
  \begin{prop}[Corol. 17.8.8 in \cite{bl98}]\label{compactinv}
  For any $C$ there is a natural isomorphism
  $$KK_G(C,\mathfrak{K}(\hu_G))\cong KK_G(C,\C).$$\end{prop}
  
\subsection{A natural transformation}\label{naturaltrans}In this section we define a natural transformation $\mathfrak{A}$ from the equivariant homotopy groups of the configuration space $\c(X,G)$ to the equivariant reduced K-homology groups of $X$. To this end we use the description of equivariant reduced K-homology as a KK-group given in Proposition \ref{KK}.

Given a $G$-${ }^*$-homomorphism $\phi:C\rightarrow\mathfrak{K}(\hu_G)$, an element in $KK_G(C,\C)$ can be assigned as
Example \ref{fundamental}. Thus if $X$ is $G$-connected, we have a map 
\begin{align*}\mathfrak{A}(X):\pi_0(\c(X,G)^G)&\longrightarrow KK^0_G(C_0(X),\C)\cong \widetilde{K}_0^G(X)\\
[\phi]&\longmapsto[(\mathfrak{K}(\hu_G),\phi,0)].\end{align*}
Notice that homotopy of points on $\c(X,G)$ correspond to a standard homotopy on Kasparov $G$-modules then the map $\mathfrak{A}(X)$ is well defined.
The map $\mathfrak{A}(-)$ is a natural transformation  from the functor $\pi_0(\c(-,G)^G)$ to $KK^0_G(-,\C)$. 
To extend this natural transformation to all $n\geq0$ we need 
a particular form of the Bott periodicity theorem.
\begin{defin}
For any $C$ and $B$ C${ }^*$-algebras, we denote the \textit{suspension of $C$} by $$SC= \{f:S^1\rightarrow C\mid f \text{ is continuous and }f(1)=0\}.$$With pointwise operations and sup norm. We denote by $S^nC$ to the $n$-th suspension of $C$,  $\underbrace{S(\cdots(S}_{n\ \text{times}}C))$.
\end{defin}
  \begin{teor}[Corol. 19.2.2 in \cite{bl98}]\label{bott}
  We have a natural isomorphism in $C$ and $B$, $$KK^1_G(C,B)\cong KK_G(C,SB).$$
  \end{teor}
 \begin{prop}
 There is a natural transformation $$\mathfrak{A}^n(X):\pi_n(\c(X,G)^G)\longrightarrow KK_G^n(C_0(X),\C)$$
 \end{prop}
 \begin{proof}
 
 For $n=0$ it was already defined. For $n=1$ consider an element $[l]\in\pi_1(\c(X,G)^G)$ with $l:S^1\longrightarrow\c(C_0(X),G)^G$. Given $f\in C_0(X)$ and $t\in S^1$, we have $l(t)[f]\in\FR_n(\hu_G)\subseteq\mathfrak{K}(\hu_G)$, for some $n\geq0$. 

As the topology is the compact-open topology, we have a continuous map $$l(-)[f]:S^1\longrightarrow\mathfrak{K}(\hu_G). $$ It is an element of $S(\mathfrak{K}(\hu_G))$. Then we have defined a continuous map 
\begin{align*}
\Omega\c(X,G)^G&\stackrel{\ \mathcal{A}\ }{\longrightarrow}\Hom^*(C_0(X),S(\mathfrak{K}(\hu_G)))\\
t&\stackrel{\ \phantom{\mathcal{A}}\ }{\longmapsto}(f\mapsto l(-)[f]).
\end{align*}
To every element $\phi\in \Hom^*(C_0(X),S(\mathfrak{K}(\hu_G)))^G$ we can associate the Kasparov module $(S(\mathfrak{K}(\hu_G)),\phi,0)\in\mathbb{E}_G(C_0(X),S(\mathfrak{K}(\hu_G)))$. 

Composing the above two maps and taking homotopy classes we have a  
homomorphism 
$$\pi_1(\c(X,G)^G)\longrightarrow KK_G(C_0(X),S(\mathfrak{K}(\hu_G))).$$
Notice that the homotopy of two paths $l_0,l_1\in \Omega\c(X,G)^G$ correspond to a standard homotopy of the Kasparov modules $$(S(\mathfrak{K}(\hu_G)),\mathcal{A}(l_0),0),\ (S(\mathfrak{K}(\hu_G)),\mathcal{A}(l_1),0)\in\mathbb{E}_G(C_0(X),S(\mathfrak{K}(\hu_G))),$$then the above map is well defined. If we suppose that $X$ is $G$-connected, using the identification given by Theorem \ref{bott},we have defined the natural transformation 

$$\mathfrak{A}^1(X):\pi_1(\c(X,G)^G)\longrightarrow KK_G^1(C_0(X),\C).$$
For every $n\geq0$ and for every $X$ (non necessary $G$-connected) is defined as the following composition

\begin{align*}
\pi_{n+1}(\c(\Sigma X,G)^G)&\longrightarrow [\Hom^*(C_0(\Sigma X),S^{n+1}(\mathfrak{K}(\hu_G)))^G]\\
                           &\longrightarrow KK_G(C_0(\Sigma X),S^{n+1}(\mathfrak{K}(\hu_G)))\\
                           &\xrightarrow{\,\,\cong\,} KK_G^{n+1}(C_0(\Sigma X),\C)\\
                           &\xrightarrow{\,\,\cong\,} KK_G^n(C_0(X),\C).
\end{align*}
where the first isomorphism is given by Theorem \ref{bott} and Prop \ref{compactinv} and the last one is given by the suspension isomorphism.
\end{proof}


The last construction gives a relation between the homotopy groups of the configuration space and the analytic construction of 
KK-theory by Kasparov.
\begin{teor}\label{s1t}
 The map $$\mathfrak{A}^n(S^0):\pi_n(\c(S^0,G)^G)\longrightarrow KK^n_G(C_0(S^0),\C)=KK^n_G(\C,\C)$$is an isomorphism for every $n\geq0$, where $S^0$ is endowed with the trivial $G$-action.
\end{teor}
\begin{proof}

First we will prove that $\mathfrak{A}^0(S^0)$ is an isomorphism. To prove surjectivity consider a Kasparov $G$-module $\alpha=(\hu_G,\phi,F)$ in $\E_G(\C,\C)$. The Hilbert space $\hu_G$ is a $\Z_2$-graded $G$-space and the map $\phi(1)$ is a projection of degree 0, which means that $\hu=\hu_G^0\oplus\hu_G^{1,op}$ and $\phi(1)=diag(P,Q)$ for projections $P$ and $Q$; The operator $F$ has the form
$\begin{pmatrix}0 &S\\T &0 \end{pmatrix}.$ The Kasparov module is $KK_G$-equivalent to $(\widetilde{P}\hu_G^0,Id_n,0)$ (for details on this argument the reader can refer to \cite[Example 17.3.4]{bl98}), where $\widetilde{P}\hu_G^0$ is a complex $n$-dimensional $G$-representation. To prove injectivity note that the map $Id_n:\C\mapsto\mathfrak{K}(\widetilde{P}\hu_G^0)$ is the $H$-map sending 1 to the identity matrix, and each of these modules constitutes different elements of $KK_G(\C,\C)$. It proves that $\mathfrak{A}^0(S^0)$ is an isomorphism. If $n$ is odd there is nothing to prove because both groups $\k_n^G(S^0)$ and $\widetilde K_n^G(S^0)$ are zero. If $n$ is even we have the commutative diagram
$$\xymatrix{\k_n^G(S^0)\ar[d]\ar[r]^{\mathfrak{A}^n(S^0)} & \widetilde{K}_n^G(S^0)\ar[d]\\
\k_0^G(S^0)\ar[r]^{\mathfrak{A}^0(S^0)} & \widetilde{K}_0^G(S^0)}$$
where the vertical arrows are Bott periodicity, and $\mathfrak{A}^0(S^0)$ is an isomorphism, hence $\mathfrak{A}^n(S^0)$ is an isomorphism also.
\end{proof}
\section{Induction structure}\label{section3}
In this section we prove that the equivariant connective K-homology has an induction structure in the sense of Definition 
\ref{induction}. 
Let $\alpha:H\rightarrow G$ be a group homomorphism (where $H$ and $G$ are finite). Let $X$ be an $H$-space such that $\Ker(\alpha)$ acts 
freely on $X$. There is a map 
\begin{align*}i_\alpha:X\longrightarrow&  \hspace{0.1cm}G\times_\alpha X\\
x\longmapsto & \hspace{0.1cm}[e,x].\end{align*}
We start when the map $\alpha$ is an inclusion. Using the induction structure in this case it can be proved that the natural transformation $\mathfrak{A}$ defined in Section \ref{naturaltrans} is an equivalence.
\subsection{Equivalence with equivariant connective K-homology}
  \begin{lema}\label{ind1}Given a subgroup $H\subseteq G$ and $X$ an $H$-CW-complex then the inclusion $i:X\longrightarrow G\times_HX$ induces an isomorphism
  $$i_*:\k_*^H(X)\longrightarrow\k_*^G(G\times_HX).$$
  \end{lema}
  \begin{proof}
The Hilbert space  $\hu_G$ can be considered as a complete $H$-universe and then $\hu_G$ is isomorphic as a $H$-module with $\hu_H$, therefore we can suppose that $$\c(X,H)^H=\bigcup_{n\geq0}\Hom^*(C_0(X),\FR_n(\hu_G))^H.$$
  We define a map 
  \begin{align*}
  i_*:\bigcup_{n\geq0}\Hom^*(C_0(X),\FR_n(\hu_G))^H & \longrightarrow \bigcup_{n\geq0}\Hom^*(C_0(G\times_HX),\FR_n(\hu_G))^G\\
  F &\longmapsto  i_*(F)=\frac{1}{|H|}\sum_{g\in G}g\cdot(F(g^{-1}\cdot f\mid_X),
  \end{align*}
  where $f\in C_0(G\times_HX)$.  On the other hand we define a map 
  \begin{align*}
  \chi:\bigcup_{n\geq0}\Hom^*(C_0(G\times_HX),\FR_n(\hu_G))^G&\longrightarrow \bigcup_{n\geq0}\Hom^*(C_0(X),\FR_n(\hu_G))^H\\
  F&\longmapsto  F\circ\mu,
  \end{align*}  
  where $\mu:C_0(X)\longrightarrow C_0(G\times_HX)$ is a continuous, $H$-equivariant map such that $\mu(f)\mid_X=f$
  and $\mu(f)|_{G\times_HX- X}=0$ for every $f\in 
  C_0(X)$.
  We have that $$\chi\circ i_*=id_{\c(X,H)^H},$$ since $f\in C_0(X)$, 
  \begin{align*}
  \chi(i_*(F))(f)&=i_*(F)(\mu(f))\\&=\frac{1}{|H|}\sum_{g\in G}g\cdot(F((g^{-1}\cdot\mu(f))|_X))\\
 & =\frac{1}{|H|}\sum_{h\in H}h\cdot F(h^{-1}\cdot f)\\&=F(f).
  \end{align*}
  Furthermore, $i_*\circ\chi= id_{\c(G\times_HX,G)^G}$ because given $f\in C_0(G\times_HX)$, 
  \begin{align*}
  i_*(\chi(F))(f)&=\frac{1}{|H|}\sum_{g\in G}g\cdot(\chi(F)((g^{-1}\cdot f)|_X))\\
 & = \frac{1}{|H|}\sum_{g\in G}g\cdot(F(\mu((g^{-1}\cdot f)|_X )))\\&=\frac{1}{|H|}\sum_{g\in G}F(g\cdot((\mu(g^{-1}\cdot f)|_X )))\\&=F(f),
 \end{align*}
  where the last equality is  a consequence of $f=\frac{1}{|H|}\sum_{g\in G}g\cdot\mu((g^{-1}\cdot f)|_X)$.
  Taking homotopy groups we obtain the desired result. 
    \end{proof}
As a consequence of the above lemma we obtain the following theorem.
 \begin{teor}\label{equivalencia}
  The functor $\k^?_*$ is naturally equivalent to $ \ku_*^?$; that is, the equivariant homotopy groups of $\c(X,G)$ are isomorphic to the equivariant reduced connective  K-homology groups of $X$ when $X$ is a finite $G$-CW-complex and $G$ is finite.
  \end{teor}
 \begin{proof}
  We already proved in Theorem \ref{s1t} that the natural transformation $\mathfrak{A}$ defined in the section above is an isomorphism when $X=S^0$ with a 
  trivial $G$-action. To prove the theorem it is enough to prove that $\mathfrak{A}$ is an isomorphism when $X=S^0\wedge G/H=(G/H)_+$ with trivial 
  $G$-action over $S^0$ and the usual $G$-action over $G/H$. It is so because we proceed by cellular induction. Consider the commutative diagram
  $$\xymatrix{
  \k^H_*(S^0)\ar[rr]^{\mathfrak{A}(S^0)}\ar[d]^{i_*}&&\widetilde{K}_*^H(S^0)\ar[d]^{i_*}\\
  \k_*^G(S^0\wedge G/H)\ar[rr]^{\mathfrak{A}(S^0\wedge G/H)}&&\widetilde{K}_*^G(S^0\wedge G/H)}$$
  where $i_*:\k^H_*(S^0)\rightarrow \k_*^G(S^0\wedge G/H)$ is the isomorphism obtained in Lemma \ref{ind1} and $i_*:\widetilde{K}^H_*(S^0)\rightarrow \widetilde{K}_*^G(S^0\wedge G/H)$
  is the isomorphism obtained from the induction structure for equivariant K-homology (see \cite{lu2002}). From this diagram we obtain that the map
  $$\mathfrak{A}(S^0\wedge G/H):\k_*^G(S^0\wedge G/H)\rightarrow \widetilde{K}_*^G(S^0\wedge G/H)$$is an isomorphism.
 \end{proof}
 As a corollary of Theorem \ref{equivalencia} we can construct a model for equivariant connective K-theory  spectrum for finite groups.

 Consider the sequence of spaces $$
 \c_n^G=\begin{cases}\c(S^n,G)&\text{ if } n>0\\
        \Omega \c(\Sigma S^n,G)& \text{ if }n=0\\
        \{pt\} &\text{ if } n<0\end{cases}.$$Define the structure maps as follows. For $n>0$, let $F\in\c_n^G=\c(S^n,G)$, we define an element 
 $$\sigma^n(F)\in\Omega\c(S^{n+1},G).$$ Notice that if $t\in S^1$ and $f\in C_0(S^{n+1})$ we have $f(t,-)\in C_0(S^n)$, then we define
 $$[\sigma^n(F)](t)[f]=F(f(t,-)).$$
It defines a continuous map
\begin{align*}
\sigma^n:\c(S^n,G)&\ \longrightarrow\Omega\c(S^{n+1},G)\\
F&\ \longmapsto\big(t\mapsto(f\mapsto F(f(t,-)))\big).
\end{align*} 
The maps $\sigma^n$ are weak $G$-homotopy equivalences because taking homotopy classes the maps $\sigma$ corresponds with the suspension isomorphism for $\k_*^G$. For $n=0$ the structre map is defined in a similar way. Then $(\c_n^G,\sigma^n)$ is a $\Omega$-$G$-spectrum.
\begin{teor}

  The $\Omega$-$G$-spectrum $(\c_n^G,\sigma^n)$ is a $G$-spectrum representing equivariant connective
 K-theory.
\end{teor}
\begin{proof}
Denote by $H_*^{\c G}$ to the reduced $G$-homology theory associated to the $\Omega$-$G$-spectrum $(\c_n^G,\sigma^n)$. We will define a natural transformation $$H_*^{\c G}\longrightarrow \k_*^G.$$ Let $(X,x_0)$ be a based $G$-CW-complex. We have a map 
\begin{align*} 
X\wedge \c_n^G& \stackrel{\ j\ }{\longrightarrow}\c(S^n\wedge X,G)\\
(x,F)& \stackrel{\ \phantom{j}\ }{\longmapsto} j(x,F)(f)=F(f(-,x))\end{align*}
for $f\in C_0(S^n\wedge X)$.

On the other hand the natural transformation $\mathfrak{A}$ defined in Section \ref{naturaltrans} gives us a well defined map 
$$\mathfrak{A}(S^n\wedge X):\pi_0(\c(S^n\wedge X,G))\longrightarrow KK_G^0(C_0(S^n\wedge X), \C),$$
and Bott isomorphism (Theorem \ref{bott}) gives us a map
$$KK_G^0(C_0(S^n\wedge X),\C)\xrightarrow{\ \cong\ }KK_G^0(C_0(X),S^n\C)\cong \widetilde{K}_n^G(X).$$
Composing the above three maps we obtain a map
$$\pi_0(X\wedge\c_n^G)\longrightarrow \widetilde{K}_n^G(X). $$
We have constructed a natural transformation $H_*^{\c G}\longrightarrow \widetilde{K}_*^G$ satisfying the conditions in Proposition \ref{connectiveversion}. Hence $H_*^{\c G}$ is naturally equivalent to $\k_*^G$ and $(\c_n^G,\sigma^n)$ is a $\Omega$-$G$-spectrum representing equivariant connective K-homology. 
\end{proof}

 \subsection{Induction structure for general homomorphisms}
To obtain the induction structure when $\alpha$ is an arbitrary map we need a lemma.
From now on we denote  the functor ${\k}_*^?$ by $\widetilde{k}_*^?$ and $\underline{k}_*^?$ by $k_*^?$. 
\begin{lema}\label{ind2}
  Let $X$ be a  $G$-CW-complex such that $N\trianglelefteq G$ acts freely in $X$. Then there is natural 
  isomorphism
  $\pi_*:\widetilde{k}_*^G(X)\longrightarrow\widetilde{k}_*^{G/N}(X/N)$ induced by the quotient map $\pi:X\longrightarrow X/N$.

\end{lema}
\begin{proof}The algebra $C_0(X/N)$ can be identified with $C_0(X)^N$, this is the algebra of continuous maps from $X$ to $\C$ that are invariant by the action of 
$N$, this algebra has a $G$-action as a subalgebra of $C_0(X)$. With this identification lets consider the natural map
\begin{align*}C_0(X)&\stackrel{\ p\ }{\longrightarrow} C_0(X)^N\\
f&\stackrel{\ \phantom{p}\ }{\longmapsto}\frac{1}{|N|}\sum_{g\in N}g\cdot f .
\end{align*}
This allows us to define a *-homomorphism 

\begin{align*}
\bigcup_{n\geq0}\Hom^*(C_0(X)^N,FR_n(\hu_{G/N})&)^{G/N}\xrightarrow{\ \ \phi\ \ }
\bigcup_{n\geq0}\Hom^*(C_0(X),FR_n(\hu_{G/N}))\\
A&\xmapsto{\ \  \ \ \ \ \ \ \  \phantom{\phi}\ } A\circ p.
\end{align*}
On the other hand $\hu_{G/N}$ can be identified with $(\hu_G)^N$, so we can suppose that $A\circ p$ is an element of $\bigcup_{n\geq0}\Hom^*(C_0(X),FR_n((\hu_G)^N))$. First we will prove that $\phi$ is $G$-equivariant.
Given an element $g\in G$, let
$\pi:G\rightarrow G/N$ be the quotient map and $f\in C_0(X)$. Then 
\begin{align*}
\left(g\cdot (A\circ p)\right)(f)&=g\left((A\circ p)(g^{-1}\cdot f )\right)g^{-1}\\
                                  &=g(A(\pi(g^{-1})\cdot (p(f))))g^{-1}\\
                                  &=\pi(g)(A(\pi(g^{-1})\cdot(p(f))))\pi(g^{-1})\\
                                  &=A(p(f)),
\end{align*}
Where the last equality is because $A$ is $G/N$-invariant.
Hence $$A\circ p\in \bigcup_{n\geq0}\Hom^*(C_0(X), FR_n((\hu_G)^N))^G.$$ We define a map \begin{align*}\bar{p}:\hu_G&\longrightarrow (\hu_G)^N\\
                          v&\longmapsto \frac{1}{|N|}\sum_{g\in N}gv.\end{align*} Now consider the following 
commutative diagram (here $i:(\hu_G)^N\to\hu_G$ is the inclusion)
\begin{equation*}\xymatrix{
\hu_G        \ar[rr]^{i\circ A(p(f))\circ \bar{p}}\ar[d]^{\bar{p}}&&\hu_G\\
(\hu_G)^N\ar[rr]^{A(p(f))}                                                              &&(\hu_G)^N\ar[u]^{i}
}
\end{equation*}

It implies that $i\circ A( p(-)) \circ\bar{p}\in \bigcup_{n\geq0}\Hom^*(C_0(X),FR_n(\hu_G))^G$. Then we have a natural map 
\begin{align*}
\chi:\c(X/N,G/N)^{G/N}&\longrightarrow\c(X,G)^G\\
A&\longmapsto i\circ(A\circ p)\circ\bar{p}.
\end{align*}
We will prove that $\chi$ induces an isomorphism on homotopy groups.

The map $X\xrightarrow{\pi} X/N$ is a $G$-covering space. As the group is finite this implies that $X=\coprod_i \widetilde{U_i}$, where each 
$\widetilde{U_i}\cong_GG\times_HU_i$,
and $U_i$ is a $G$-contractible open set where $H$ acts trivially, $\widetilde{U_i}\simeq_G G$ and $\pi(\widetilde{U_i})\simeq_{G/N}G/N$.

 The map $\chi_*:{k}_*^{G/N}(G/N)\rightarrow{k}_*^G(G)$ is an isomorphism because we have the following commutative diagram
$$\begin{diagram}
   \node{{k}_*^G(G)}\arrow{s,r}{\chi_*}\node{{k}_*(pt)}\arrow{w,t}{i_*}\arrow{sw,r}{i_*}\\
   \node{{k}_*^{G/N}(G/N)}
  \end{diagram}$$
and both $i_*$ are isomorphisms by Lemma \ref{ind1}. Now, the result for general $X$ follows by an inductive argument using the disjoint union axiom and the decomposition $X\coprod_iU_i$.
   \end{proof}
Using the above results we can derive an induction structure. 
\begin{teor}\label{iluck}Given $\alpha:H\rightarrow G$ such that $\Ker(\alpha)$ acts freely in $X$, the map $i:X\rightarrow G\times_\alpha X$ induces a
natural isomorphism$$i_*:\widetilde{k}_*^H(X)\rightarrow \widetilde{k}_*^G(G\times_\alpha X).$$\end{teor}
\begin{proof} If $\alpha:H\rightarrow G$ is a group homomorphism, $\alpha$ can be obtained as the composition 
$$H\xrightarrow{\ \alpha\ }\alpha(H)\xrightarrow{\ i\ }G,$$ so $G\times_HX\cong_GG\times_i(\alpha(H)\times_\alpha X)$, 
and this allows us to  obtain the following isomorphisms
$$\k_*^G(G\times_\alpha X)\cong\k_*^G(G\times_{i}(\alpha(H)\times_\alpha X)).$$
On the other hand Lemma \ref{ind1} implies

$$\k_*^G(G\times_{i}(\alpha(H)\times_\alpha X))\cong\k_*^{\alpha(H)}(\alpha(H)\times_\alpha X)$$
From the homomorphism $\alpha:H\rightarrow \alpha(H)$ we obtain an isomorphism \\$\bar{\alpha}:H/\Ker(\alpha)\rightarrow \alpha(H)$, and then 
$$\k_*^{\alpha(H)}(\alpha(H)\times_\alpha X)\cong
\k_*^{H/\Ker(\alpha)}(H/\Ker(\alpha)\times_\pi X)$$where $\pi:H\rightarrow H/\Ker(\alpha)$ is the quotient map. Finally Lemma 
\ref{ind2} implies
$$\k_*^{H/\Ker(\alpha)}(H/\Ker(\alpha)\times_\pi X)\cong\k_*^H(X).$$
\end{proof}

We will verify the properties in Definition \ref{induction}. For this we use the fact that the map defined to obtain the above isomorphism is the
 \textit{invariantization}. 
\begin{enumerate}
 \item \textbf{Compatibility with the boundary homomorphism}.
  If $p:E\rightarrow B$ is a $G$-quasifibration with fibre $F$, we have a connecting morphism 
  $$\partial_n^G:\pi_n^G(B,b)\longrightarrow \pi_{n-1}^G(F,f)$$ defined in the following way. If $[\varphi]\in\pi_n^G(B,b)$, this element can be 
  viewed as an element of $\pi_n^G(E,p^{-1}(b),x_0)\cong\pi_n^G(F,f)$ (since $p$ is a quasifibration), and the homotopy class of the  image of the map
   $\varphi$ restricted to $\partial D\cong S^{n-1}$ by the above identification that we denote by 
  $$\widetilde{\varphi}:(D^n,\partial D^n)\longrightarrow (E^G,(p^{-1}(b))^G,x_0),$$ can be viewed as an element of 
  $\pi_{n-1}^G(F,f)$. The above argument implies that the connecting morphism in this case is given by a restriction. The compatibility
  with the boundary map follows from the fact that the invariantization commutes with restrictions.

 \item\textbf{Functoriality.}
  This property follows from the fact that taking invariantization is \textit{transitive}, that means, if we have homomorphisms 
  $\alpha:H\rightarrow G$ and $\beta:G\rightarrow K$ then the invariantization map defined from  $C_0(K\times_{\beta\circ\alpha} X)$ to
  $C_0(X)$ with $X$ a $H$-space is the composition of the invariantizations defined from $C_0(K\times_{\beta\circ\alpha} X)$ to 
  $C_0(G\times_{\alpha} X)$ and  from $C_0(G\times_{\alpha} X)$ to $C_0(X)$.

 \item\textbf{Compatibility with conjugation.}
  This property follows from the fact that to conjugate and later take the invariantization is the same as to take the 
  in\-va\-rian\-ti\-za\-tion without
  conjugate.
\end{enumerate}
The above argument and Theorem \ref{iluck} proves the following theorem.
\begin{teor}
 The functor $k_*^?$ is an equivariant homology theory in the sense of Definition \ref{induction}.
\end{teor}
\begin{remark}
The induction structure for equivariant connective K-homology is a new result. Note that for periodic K-homology is possible to obtain the induction structure from the induction of the representation rings of the corresponding groups, but in the classical definition of connective K-homology (as the homotopy groups of the connective cover of the K-theory spectrum) we cannot use this correspondence. For our purposes is necessary because we want to apply the equivariant Chern character coming from of Theorem \ref{chern}.
\end{remark}
\section{The algebra $\F^q_G(X)$}\label{section4}
As is noted in \cite{gr1996} given a cohomology theory $\mathsf{H}$ defined on the category of orbifolds one can associate a commutative and cocommutative Hopf algebra
$$S=\bigoplus_{n\geq0}\mathsf{H}(X^n/\mathfrak{S}_n),$$
in \cite{gr1996} is described a generator set of $S$ and $S$ is identified with a Fock space. The case of orbifold K-theory is studied by Segal in \cite{segal1996}. He consider a $\Z\times\Z/2\Z$-graded Hopf algebra

 $$\F^q(X)=\bigoplus q^n K_{\S_n}^*(X^n)\otimes\C,$$
where $q$ is a formal variable counting the $\Z$-grading and $K_G^*(-)=K_G^0(-)\oplus K_G^1(-)$. In \cite{segal1996} is established an isomorphism between the completion at the augmentation ideal of  $\F^q(X)$ 
and the homology with complex coefficients of the configuration space $\c(X)$. More precisely.
\begin{teor}[\cite{segal1996}]Let $X$ be a $\spinc$-manifold and $H_*(\c(X);\C)$ is the complex homology endowed with the Pontryagin product. If we denote by $\widehat{(\ )}$ the completion at the augmentation ideal, then there is a $\Z\times\Z/2\Z$-graded Hopf-algebra isomorphism$$\widehat{\F^q(X)}\cong H_*(\c(X);\C).$$
\end{teor}
The goal of this section is to obtain an equivariant generalization of the above theorem. 

 Let $X$ be a topological space endowed with an action of a finite group $G$. 
We consider the \textit{wreath} product $G_n=G\wr \S_n$ which is a 
semidirect product of the $n$-th direct product $G^n$ of $G$ and the symmetric group $\S_n$. If $G$ acts on $X$ there is a natural action 
of the group $G_n$ on $X^n$ induced by the actions of $G^n$ and $\S_n$ on $X^n$. \begin{defin} Define the $\Z\times\Z/2\Z$-graded complex vector space $$\F_G^q(X)=\bigoplus_{n\geq0}q^nK_{G_n}^*(X^n)\otimes\C.$$
\end{defin} Wang in \cite{wang2000} shows that $\F^q_G(X)$ admits a natural $\Z\times\Z/2\Z$-graded Hopf algebra 
structure that we describe in the following lines.

First we recall the induction structure on equivariant K-theory.

\begin{defin}\label{indhaces}
Let $G$ and $H$ be finite groups, and $\alpha:H\rightarrow G$ a group homomorphism. Let $X$ be a $G$-space and $E\xrightarrow{\ p\ }X$ be an $H$-vector bundle over $X$. According to Definition \ref{induction} one can consider the map $G\times_\alpha E\xrightarrow{\ \ \bar{p}\ \ }G\times_\alpha X$. In \cite[Lemma 6]{wang2000} it is proved that the above map carries a natural $G$-vector bundle structure over $X$. Passing to isomorphism classes one can define a map
$$ind_\alpha:K_H^*(X)\rightarrow K_G^*(X).$$
\end{defin}
Now we are ready to define the product. If $\alpha:G_n\times G_m\rightarrow G_{n+m}$ is the natural inclusion, define a multiplication $\cdot$ on $\F_G^q(X)$ by a composition of the induction map and
the Kunneth isomorphism $q$: 
\begin{align*}
K_{G_n}^*(X^n)\otimes\C)\otimes (K_{G_m}^*(X^m)\otimes\C)&\xrightarrow{\ \ q\ \ }
K_{G_n\times G_m}^*(X^{n+m})\otimes\C\\&\xrightarrow{Ind_\alpha}
K_{G_{n+m}}^*(X^{n+m})\otimes\C.
\end{align*}
 We denote by 1 the unit in $K_{G_0}(X^0)\otimes\C\cong\C.$ 

The comultiplication $\Delta$ on $\F_G^q (X)$,is the
composition of the inverse of the Kunneth isomorphism and the restriction from
$G_n$ to $G_k\times G_{n-k}$:
\begin{align*}
K_{G_n}^*(X^n)\otimes\C&\longrightarrow
\bigoplus_{m=0}^nK_{G_m\times G_{n-m}}^*(X^n)\otimes\C\\&\xrightarrow{q^{-1}}
\bigoplus_{m=0}^nK_{G_m}^*(X^m)\otimes K_{G_{n-m}}^*(X^{n-m})\otimes\C.
\end{align*}
We define the counit $\epsilon:\F_G(X)\rightarrow\C$ by sending $K_{G_n}^*(X^n)$ ($n> 0$) to 0 and
$1\in K_{G_0}^*(X^0)\cong\C$ to 1.

\begin{teor}[Thm. 2 in \cite{wang2000}] With the operations defined as above, $\F_G^q(X)$ becomes a Hopf
algebra.
\end{teor}

It is possible to describe $\F_G^q(X)$ as a supersymmetric $\Z\times\Z/2\Z$-graded algebra. For this end we will use a version of Chern character in equivariant K-theory given in \cite{a-s1989}.

Note that $K_G^*(pt)\otimes\C$ is
isomorphic to the ring $class_\C(G)$ of class functions on $G$. The bilinear map $\star$ induced
from the tensor product $$K_G^*(pt)\otimes K_G^*(X)\rightarrow K_G^*(X)$$
gives rise to a natural $K_G^*(pt)$-module structure on $K_G^*(X)$. Thus $K_G^*(X)\otimes\C$ naturally
decomposes into a direct sum over the set of conjugacy classes $G_*$ of $G$. The
next theorem \cite{a-s1989} gives a description of each term in the direct sum.
\begin{lema}\label{ascharacter}There is a natural $\Z/2\Z$-graded isomorphism
$$\phi:K_G^*(X)\otimes\C\longrightarrow\bigoplus_{[g]}K^*(X^g/Z_G(g))\otimes\C,$$
where $Z_G(g)$ denotes the centralizer of $g$ in $G$.
\end{lema}
Applying lemma \ref{ascharacter} to each term of $\F_G^q(X)$ one can derive a decomposition theorem. First we recall the notion of $\Z\times\Z/2\Z$-graded supersymmetric algebra.
\begin{defin}
Let $A=A^0\oplus A^1$ be a $\Z\times\Z/2\Z$-graded complex vector space, we define the $\Z\times\Z/2\Z$-graded supersymmetric algebra $\s(A)$ as the tensor product of the symmetric algebra $S(A^0)$ and the exterior algebra $\Lambda(A^1)$.
\end{defin}
\begin{teor}[Thm. 3 in \cite{wang2000}]\label{description} As a $\Z\times\Z/2\Z$-graded Hopf algebra, $\F_G^q(X)$ is isomorphic to the supersymmetric algebra 
$\s(\bigoplus_{n\geq1}q^nK_G^*(X))$.
\end{teor}


\subsection{Hopf spaces}
Note that the configuration space $\c(X,G)$ has a natural structure of Hopf space given by `putting together' the configurations. More formally we have the following product.
\begin{defin}
Given $$F_1\in Hom^*(C_0(X),FR_n(\hu_G))^G$$ and $$F_2\in Hom^*(C_0(X),FR_m(\hu_G))^G,$$ we define 
$$F_1\cdot F_2\in Hom^*(C_0(X),FR_{n+m}(\hu_G))^G$$
$$(F_1\cdot F_2)(f)=F_1(f)\oplus F_2(f),$$
where $\oplus$ is the \textit{external} direct sum of operators, which is the composition 
$$\hu_G\xrightarrow{\ F_1(f)\oplus F_2(f)\ }\hu_G\oplus\hu_G\xrightarrow{\ \zeta\ }\hu_G,$$
where $\zeta$ is a $G$-isomorphism of complete $G$-universes. With this operation $\c(X,G)^G$ 
becomes a homotopy associative Hopf space with unit.
\end{defin}

When a Hopf space $Y$ is pathwise connected we have a way to relate the Hopf algebra $H_*(Y;\C)$ with the $\Z$-graded complex vector space $$\pi_*(Y;\C)=\bigoplus_{n\geq0}\pi_i(Y;\C)$$ given by a theorem due to Milnor and Moore. Recall that $\widehat{(-) }$ denotes the completion with respect to the $I$-adic topology, when $I$ is the augmentation ideal i.e 
$I=\Ker(\epsilon)$ where $\epsilon$ is the counit of the Hopf algebra $\s(\pi_*(Y;\C))$. (For an explanation on completions see 
  \cite[Chapter 10]{a69}).
  \begin{teor}[Thm. of the Appendix in \cite{mm65}]\label{mm}
  If $Y$ is a pathwise connected homotopy associative Hopf space with unit, and 
  $\lambda:\pi_*(Y;\C)\rightarrow H_*(Y;\C)$ is the Hurewicz morphism viewed as a morphism of $\Z$-graded Lie algebras, then the induced 
  morphism $\bar{\lambda}:\widehat{\s}(\pi(Y;\C))\rightarrow H_*(Y;\C)$ is an isomorphism of Hopf algebras.
  \end{teor}

In order to consider the case when $Y$ is not pathwise connected we need to introduce a theorem due to Cartier. 

Let $A$ be a $\Z$-graded Hopf algebra. Let $\mathfrak{g}$ be the set of primitive elements, these are elements $x$ in $A$
such that
$$\Delta(x)=x\otimes1+1\otimes x\text{, }\epsilon(x)=0.$$
Then $\mathfrak{g}$ is a Lie algebra with the bracket $[x,y]=xy-yx$, and we can consider
its enveloping algebra $U(\mathfrak{g})$ viewed as a Hopf algebra.
Let $\Gamma$ be the set of `group-like' elements, that is the elements $g$ in $A$ such
that
$$\Delta(g)=g\otimes g\text{, }\epsilon(g)=1.$$
For the multiplication in A, the elements of $\Gamma$ form a group, where the inverse
of $g$ is $S(g)$ (here $S$ is the antipode in $A$). We can consider the group
algebra $\C\Gamma$ viewed as a Hopf algebra.
Furthermore for $x$ in $\mathfrak{g}$ and $g$ in $\Gamma$, we have that ${ }^gx=gxg^{-1}$ belongs
to $\mathfrak{g}$. Hence the group $\Gamma$ acts on the Lie algebra $\mathfrak{g}$ by conjugation and therefore on its enveloping algebra $U(\mathfrak{g})$. We define the twisted tensor product $\Gamma\ltimes U(\mathfrak{g})$ as the
tensor product $U(\mathfrak{g})\otimes \C\Gamma$ with the multiplication given by
$$(u\otimes g)\cdot(u'\otimes g')=u\cdot { }^gu'\otimes gg'.$$
\begin{teor}\cite[Thm. 3.8.2]{cartier2007}\label{cartier}
Assume that $A$ is a cocommutative $\Z$-graded Hopf algebra. Let
$\mathfrak{g}$ be the space of primitive elements, and $\Gamma$ the group of `group-like' elements
in $A$. Then there is an isomorphism of $\Gamma\ltimes U(\mathfrak{g})$ onto $A$ as $\Z$-graded Hopf algebras,
inducing the identity on $\Gamma$ and $\mathfrak{g}$.
\end{teor} 

\subsection{The homology of the $G$-fixed points of the configuration space}

 Now we can consider $H_*(\c(X,G)^G;\C)$ with the product induced by the 
Hopf-space structure of $\c(X,G)^G$.
\begin{prop}\label{iso1}Let $X$ be a finite $G$-CW-complex, if $X$ is $G$-connected we have an isomorphism 
 $$H_*(\c(X,G)^G;\C)\cong \widehat{{\s}}(\widetilde{k}_*^G(X)\otimes\C).$$
 
For an arbitrary $X$ we have an isomorphism
 $$H_*(\Omega\c(\Sigma X,G)^G;\C)\cong \widehat{{\s}}(\widetilde{k}_*^G(X)\otimes\C).$$
\end{prop}
\begin{proof}
First suppose that $X$ is $G$-connected. By Theorem \ref{mm} there is an
isomorphism of $\Z$-graded Hopf algebras $$H_*(\c(X,G)^G;\C)\cong\widehat{{\s}}(\pi_*(\c(X,G)^G;\C)),$$
and Theorem \ref{equivalencia} gives the desired isomorphism
$$\widehat{{\s}}(\pi_*(\c(X,G)^G;\C))\cong\widehat{{\s}}(\widetilde{k}_*^G(X)\otimes\C).$$

For a general $X$ according to the definition of our theory for a finite $G$-CW-complex we have that $\Omega(\c(\Sigma X,G))$ is the configuration space in this case. The space $\Omega(\c(\Sigma X,G))^G$ is a non-connected Hopf space.
To describe its homology we will use Theorem \ref{cartier}. By the grading we have that the group-like elements  in $H_*(\Omega(\c(\Sigma X,G))^G;\C)$ correspond with $H_0(\Omega(\c(\Sigma X,G))^G;\C)$ and by Milnor-Moore 
\cite{mm65} the primitive elements correspond to $\pi_*(\Omega(\c(\Sigma X,G))^G_0)\otimes\C$, where $\Omega(\c(\Sigma X,G))^G_0$ is the connected component of the identity. By Theorem \ref{cartier} we have an isomorphism of $\Z\times\Z/2\Z$-graded Hopf algebras
$$H_*(\Omega(\c(\Sigma X,G))^G;\C)\cong H_0(\Omega(\c(\Sigma X,G))^G;\C)\ltimes U(\pi_*(\Omega(\c(\Sigma X,G))^G_0)\otimes\C).$$
As $H_0(\Omega(\c(\Sigma X,G))^G)$ and $\pi_*(\Omega(\c(\Sigma X,G))^G_0)$ are abelian and the action is by conjugation, this action of $H_0(\Omega(\c(\Sigma X,G))^G;\C)$ in $\pi_*(\Omega(\c(\Sigma X,G))^G_0)$ is trivial. Then we have an isomorphism of $\Z\times\Z/2\Z$-graded Hopf algebras
\begin{align*}
H_*(\Omega(\c(\Sigma X,G))^G;\C)&\cong H_0(\Omega(\c(\Sigma X,G))^G;\C)\otimes U(\pi_*(\Omega(\c(\Sigma X,G))^G_0)\otimes\C)\\
&\cong H_0(\Omega(\c(\Sigma X,G))^G;\C)\otimes\widehat{\s}(\pi_*(\Omega(\c(\Sigma X,G))^G_0)\otimes\C)\\
&\cong \widetilde{k}_0^G(X)\otimes \widehat{\s}(\widetilde{k}_*^G(X)\otimes\C)\\
&\cong \widehat{\s}(\widetilde{k}_*^G(X)\otimes\C).\end{align*}
\end{proof}
In order to relate the homology of the $G$-fixed point space of $\c(X,G)$ with $\F_G^q(X)$ we need to use an equivariant version of Chern character for equivariant  homology theories due to L\"uck in \cite{lu2002}.
We will apply this Chern character to $k_*^?$.
 
For a subgroup $H\subseteq G$ we denote by $N_G(H)$ the normalizer of $H$ in $G$. Let $H\cdot Z_G(H)$ be the subgroup of $N_G(H)$ consisting of elements of the form $hc$ for
$h\in H$ and $c\in Z_G(H)$. Denote by $W_G(H)$ the quotient $N_G(H)/H\cdot Z_G(H)$.

Taking characters yields an isomorphism of rings 
$$\chi:R_\C(H)\otimes_\Z\C\xrightarrow{\ \cong\ }class_\C(H).$$
We denote by $H_*^G(-;\C)$ the Bredon homology with coefficients in the complex class function ring. 
Given a finite cyclic group $C$, there is the idempotent $\theta^C_C\in class_\C(C)$ which assigns 1 to a generator of $C$ and 0 to the 
other elements. This element acts on $H_n^G(\ast;\C)=class_\C(G)$. The image $im(\theta_C^C:H_n^G(\ast;\C)\rightarrow H_n^G(\ast;\C))$ of the map given by 
multiplication with the idempotent $\theta_C^C$ is a term of the direct sum in $H_n^G(\ast;\C)$ and will be denoted by $\theta_C^C\cdot H_n^G(\ast;\C)$.

\begin{teor}[Thm.0.3 in \cite{lu2002}]\label{chern} Given an equivariant homology theory $\H_*^?$ with coefficients in the complex 
class function ring, for any group $G$ and any $G$-CW-complex $X$, let $J$ be the set of conjugacy classes $(C)$ of finite cyclic 
subgroups $C$ of $G$. Then there is an isomorphism of homology theories
$$ch_*^?:\mathscr{BH}_*^?\xrightarrow{\ \cong\ }\H_*^?$$
such that
\begin{multline*}\mathscr{BH}_*^G(X;\C) = \\
\bigoplus_{p+q=n}\bigoplus_{(C)\in J}H_p(X^C/Z_G(C);\C)\otimes_{\C[W_G(C)]}
im(\theta_C^C:\H_q^C(\ast;\C)\rightarrow\H_q^C(\ast;\C)).\end{multline*}
\end{teor}
Using the above theorem we obtain the following result.
\begin{teor}\label{4.1} Let $X$ be a $G$-CW-complex. There is a natural isomorphism of $\Z$-graded complex vector spaces
$${k}_*^G(X)\otimes\C\cong {H}_*^G(X;\C)\otimes\C[q] \cong {K}_*^G(X)\otimes\C[q]. $$
\end{teor}
\begin{proof}
In the case of $k_*^G$, $im(\theta_C^C)=k_q^{e}(pt)\otimes\C\cong\C$, then the Chern character reduces to
$$k_n^G(X)\otimes\C\cong H_n^G(X;\C)\oplus H_{n-2}^G(X;\C)\oplus\cdots.$$
Taking the sum over $n\in\N$ we obtain a graded complex vector space isomorphism:
$$k_*^G(X)\otimes\C\cong H_*^G(X;\C)\otimes\C[q] \cong K_*^G(X)\otimes\C[q]. $$
where the last isomorphism is obtained using Theorem \ref{chern} applied to the equivariant homology theory $K_*^?$ in a similar way as 
we do for $k_*^?$.
\end{proof}

Finally we find an isomorphism from $H_*(\C(X,G);\C)$ to $\widehat{\F_G^q(X)}$ when $X$ is an even dimensional $G$-$\spinc$-manifold. First we recall \textit{Poincar\'e duality} for equivariant K-theory.

  \begin{teor}\label{poinca}
  Let $M$ be a $n$-dimensional $G$-$\spinc$-manifold. Then there exists an isomorphism
  $$D:K_G^*(M_+)\longrightarrow K^G_{n-*}(M).$$
  \end{teor}

Applying Theorem \ref{poinca} and Theorem \ref{cartier} we can obtain the main result of the section.
\begin{teor}
Let $M$ be a even dimensional $G$-$\spinc$-manifold. We have an isomorphism of $\Z$-graded Hopf algebras
$$H_*(\c(M,G)^G;\C)\cong\widehat{\mathfrak{F}_G^q(M)}.$$
\end{teor}
\begin{proof}
As $M$ is  a $G$-$\spinc$ manifold we can use Theorem 
\ref{poinca}
and obtain the following isomorphism of $\Z\times\Z/2\Z$-graded Hopf algebras 
$$\widehat{{\s}}(\widetilde{k}_*^G(M)\otimes\C)\cong \widehat{{\s}}(\bigoplus_{n\geq1}q^n\widetilde{K}_*^G(M)\otimes\C)\cong \widehat{{\s}}(\bigoplus_{n\geq1}q^n\widetilde{K}_G(M_+))\cong 
\widehat{{\s}}(\bigoplus_{n\geq1}q^nK_G(M)).$$
Combining Proposition \ref{iso1}, Theorem \ref{4.1}  and Theorem \ref{description} we obtain
$$H_*(\c(M,G)^G;\C)\cong\widehat{\mathfrak{F}_G^q(M)}.$$
\end{proof}
\begin{ejem}
For $X=S^0$ we have $$\Omega(\c(\Sigma(S^0),G))\simeq BU_G.$$ Applying the above discussion to this H-space we conclude that
\begin{align*}
H_*((BU_G)^G;\C)\cong &R(G)\otimes\widehat{\s}(\pi_*((BU_G)^G)\otimes\C)\\
                             \cong& R(G)\otimes\widehat{\s}(\bigoplus_{n\geq0}R(G_n)\otimes\C)\\
                             \cong& \widehat{\s}(\bigoplus_{n\geq0}R(G_n)\otimes\C).
\end{align*}
Summarizing, we have an isomorphism 

 $$H_*((BU_G)^G;\C)\cong\widehat{\F_G^q(S^0)}=\widehat{\s}(\bigoplus_{n\geq0}R(G_n)\otimes\C).$$
\end{ejem}
We also have 
$$H_*((BU_G)^G;\C)\cong\widehat{\s}(\bigoplus_{n\geq0}R(G_n)\otimes\C)\cong\C[[\sigma_1^1,\ldots,\sigma_1^{k_1},\sigma_2^1,\ldots]]$$
where $\{\sigma_i^1,\cdots,\sigma_i^{k_i}\}$ is a complete set of non isomorphic irreducible representations of $G_i$. We expect
that the elements $\sigma_i^k$ correspond in some sense with duals of $G$-equivariant Chern classes.
\section{Acknowledgments}
The author would like to express his gratitude to Prof. Bernardo Uribe and Prof. Graeme Segal who read the 
manuscript and made a number of important suggestions.  The author also thanks the referee for his comments and suggestions which helped to improve the presentation of this paper
\bibliographystyle{alpha}
\bibliography{phd}

\end{document}